\documentclass[10pt]{article}

\usepackage{amsmath,amsthm,latexsym,amssymb,wasysym,newcent}
\usepackage{color}

\usepackage[margin=3cm]{geometry}

\usepackage{varioref}
\usepackage{hyperref}
\usepackage{cleveref}

\usepackage{bbm}
\usepackage{xargs}
\usepackage{enumerate}
\usepackage[textsize=footnotesize]{todonotes} 

\newtheorem{theorem}{Theorem}[section]
\newtheorem{proposition}[theorem]{Proposition}
\newtheorem{lemma}[theorem]{Lemma}
\newtheorem{corollary}[theorem]{Corollary}
\newtheorem{remark}[theorem]{Remark}
\newtheorem{definition}[theorem]{Definition}
\newtheorem{example}[theorem]{Example}

\newtheorem{assumption}{\textbf{H}\hspace{-2pt}}
\Crefname{assumption}{\textbf{H}\hspace{-2pt}}{\textbf{H}\hspace{-2pt}}
\crefname{assumption}{\textbf{H}}{\textbf{H}}

\allowdisplaybreaks

\allowdisplaybreaks

\newcommand{\coint}[1]{\left[#1\right)}
\newcommand{\ocint}[1]{\left(#1\right]}
\newcommand{\ooint}[1]{\left(#1\right)}
\newcommand{\ccint}[1]{\left[#1\right]}

\def\E{\mathbb{E}}
\def\P{\mathbb{P}}

\def\1{\mathds{1}}
\def\eqdef{:=}

\def\Uset{\mathsf{U}}
\def\Usigma{\mathcal{U}}
\def\Rset{\mathsf{S}}
\def\Rsigma{\mathcal{S}}

\def\rme{\mathrm{e}}
\def\rmd{\mathrm{d}}
\def\nset{\mathbb{N}}
\def\rset{\mathbb{R}}

\def\cmoment{c_{0}}
\def\cmomentone{c_{1}}
\newcommand{\indi}[1]{\mathbbm{1}_{{#1}}}

\newcommand{\ps}[2]{\left\langle #1, #2 \right\rangle}

\newcommandx{\as}[1][1=P]{\ensuremath{#1\, -\mathrm{a.s.}}}
\newcommand{\fraca}[2]{#1/#2}

\def\param{\theta}

\newcommandx{\set}[2]{\{ {#1} \, ; \, {#2}\}}

\newcommand{\boule}[1]{\operatorname{B}_{#1}}

\newcommandx{\CPE}[3][1=]{{\mathbb E}^{#1}\left[\left. #2 \, \right| #3 \right]} 

\def\bx{\mathbf{x}}
\def\Param{\Theta}

\newcommandx{\tgenerator}[4][1=\lambda,2=\beta,3=p,4=\theta]{\bar{\operatorname{L}}^{#1}_{#2} V_{#3}(#4)}
\newcommandx{\generator}[5][1=\lambda,2=\beta,3=x,4=p,5=\theta]{\operatorname{L}^{#1}_{#2,#3} V_{#4}(#5)}

\newcommand{\tzeta}[5]{\tilde{Y}^{#1}_{#3,#2}(#4,#5)}
\newcommand{\tZ}[4]{L^{#1}_{#3,#2}(#4)}
\newcommand{\bZ}[3]{\overline{L}^{#1}_{#2,#3}}
\newcommandx{\tU}[3][1=n,2=t,3=\lambda]{U_{#1,#2}^{#3}}
\newcommandx{\tD}[3][1=n,2=t,3=\bx]{\Delta_{#1,#2}(#3)} 
\def\bW{\mathbf{W}}

\begin{document}


\title{On stochastic gradient Langevin dynamics with dependent data streams:
the fully non-convex case
\thanks{All the authors were supported by The Alan Turing Institute, London under the EPSRC grant EP/N510129/1. N. H. C. and M. R. also enjoyed the support of the NKFIH (National Research, Development and Innovation Office, Hungary) grant KH 126505 and the ``Lend\"ulet'' grant LP 2015-6 of the Hungarian Academy of Sciences. Y. Z. was supported by The Maxwell Institute Graduate School in Analysis and its Applications, a Centre for Doctoral Training funded by the UK Engineering and Physical Sciences Research Council (grant EP/L016508/01), the Scottish Funding Council, Heriot-Watt University and the University of Edinburgh. We thank the Alan Turing Institute, London, UK; the R\'enyi Institute, Budapest, Hungary and the \'Ecole Polytechnique, Palaiseau, France for hosting research meetings of the authors.}}

\author{N. H. Chau\thanks{Osaka University, Japan.} \and
\'E. Moulines\thanks{Centre de Math\'ematiques Appliqu\'ees, UMR 7641, Ecole Polytechnique, France}
\and M. R\'asonyi\thanks{Alfr\'ed R\'enyi Institute of Mathematics, 1053 Budapest, Re\'altanoda utca 13--15,
Hungary\newline E-mail: rasonyi.miklos@renyi.hu} \and S. Sabanis\thanks{School of Mathematics, The University of Edinburgh and
The Alan Turing Institute, UK.} \and Y. Zhang\thanks{School of Mathematics, The University of Edinburgh, UK.}}

\date{\today}

\maketitle

\begin{abstract}
We consider the problem of sampling from a target distribution, which is \emph
{not necessarily logconcave}, in the context of empirical risk minimization and stochastic optimization as
presented in \cite{raginsky2017non}. Non-asymptotic analysis results are established in the
$L^1$-Wasserstein distance for the behaviour of Stochastic Gradient Langevin Dynamics (SGLD)
algorithms. We allow the estimation of gradients to be performed even in the presence of \emph{dependent} data streams.
Our convergence estimates are sharper and \emph{uniform} in the number of iterations, in contrast to those
in previous studies.
\end{abstract}

\smallskip


\smallskip


\section{Introduction}
In this paper, the problem of approximate sampling from a target distribution
\begin{equation}
\label{eq:definition-pi-beta}
\pi_{\beta}(\theta) \wasypropto \exp(-\beta U(\theta)) \rmd \theta
\end{equation}
is investigated, where $\theta \in \rset^d$, $\beta>0$, and the function $U:\rset^d \to \rset_{+}$ is differentiable, $\nabla U$ is Lipshitz-continuous, and $U$ satisfies a certain dissipativity condition.
If $U$ has a unique minimizer $\theta^{*}$ then sampling from \eqref{eq:definition-pi-beta} with a large $\beta${}
amounts to finding $\theta^{*}$.

It is well-known that \eqref{eq:definition-pi-beta} is the stationary law of the Langevin
stochastic differential equation
\begin{equation}\label{sde}
\rmd L_t = - \nabla U(L_t) \rmd t + \sqrt{2\beta^{-1}} \rmd B_t,
\end{equation}
where $B$ is a the standard Brownian motion in $\rset^d$ and $\beta>0$ is the so-called inverse temperature parameter.
Euler discretizations of \eqref{sde} lead to the extensively studied unadjusted Langevin algorithm. When only estimates for
the gradient $\nabla U$ are available, we arrive at the
Stochastic Gradient Langevin Dynamics (SGLD) algorithm (\eqref{nab} below), introduced in \cite{welling:teh:2011}, which is the focus of
our interests in the present article.

Imagine that we wish to tune the parameter $\theta$ of some software optimally so as to minimize $U(\theta)${}
which is the expectation of a given cost function depending on $\theta$ and on an observed random data
sequence whose law is unknown (and might slowly change over time). In such a situation our optimization must be data-driven
and one may use e.g.\ fixed gain stochastic gradient algorithms, see \cite{chau:kumar:rasonyi:sabanis:2019}.
In a nonconvex setting, however, there can be several local minima.
By injecting extra noise, SGLD is a powerful tool for solving such problems, see Section \ref{sec:application} below
for a more thorough discussion.


For an i.i.d.\ data sequence, the remarkable study \cite{raginsky2017non} provided theoretical guarantees in the form
of non-asymptotic convergence estimates for SGLD in the quadratic Wasserstein distance.
The purpose of the present paper is to significantly sharpen these estimates by providing optimal
rates in terms of the stepsize, using another metric, for the first time in the literature. We refer to
\cite{raginsky2017non} for further details about this method of optimization
in the big data context. We stress, however, the applicability of SGLD also in the context of online parameter optimization
where dependent data is commonly encountered.

Non-asymptotic convergence rates of Langevin dynamics based algorithms for approximate sampling of
log-concave distributions have been intensively studied in recent years,
starting with \cite{dalalyan:tsybakov:2012,dalalyan:2017}. This was followed by
\cite{dalalyan2019user,durmus:moulines:highdimULA,durmus:moulines:2017,cheng2018convergence, Brosse2017Tamed}  amongst others.

Relaxing log-concavity is a more challenging problem. In \cite{majka2018non}, the log-concavity assumption is
replaced by a ``monotonicity at infinity'' condition, convergence rates are obtained
in $L^1$- and $L^2$-Wasserstein distances.
In a similar setting, \cite{cheng2018sharp} analyzes sampling errors in the $L^1$-Wasserstein distance for
both overdamped and underdamped Langevin MCMC.

Our starting point is \cite{raginsky2017non}, where a dissipativity condition is assumed and convergence rates are obtained in
the $L^2$-Wasserstein distance. Moreover, a clear and strong link between sampling via SGLD algorithms and non-convex
optimization is highlighted. One can further consult \cite{xu2018global,dalalyan2017further} and references therein.

In the present paper, we impose the same dissipativity condition as in \cite{raginsky2017non}. Using the $L^1$-Wasserstein
metric, we obtain sharper estimates and allow for possibly dependent data sequences. The key new idea
is comparing the SGLD algorithm to a suitable auxiliary continuous time processes inspired by (\ref{sde}) and
then relying on contraction results developed in \cite{eberle:guillin:zimmer:Trans:2019} for \eqref{sde}.

\paragraph{Notations and conventions.} Let $(\Omega,\mathcal{F},\P)$ be a probability space.
We denote by $\E[X]$  the expectation of a random variable $X$.
For $1\leq p<\infty$, $L^p$ is used to denote the usual space of $p$-integrable real-valued random variables.
Fix an integer $d\geq 1$. For an $\rset^d$-valued random variable $X$, its law on $\mathcal{B}(\rset^d)$ (the Borel sigma-algebra of $\rset^d$) is denoted by $\mathcal{L}(X)$. Scalar product is denoted
by $\ps{\cdot}{\cdot}$, with $|\cdot|$ standing for the
corresponding norm (where the dimension of the space may vary depending on the context).
For  $r \in \rset_+$, denote by $\boule{r}$ the closed ball centered at $0$  with radius $r$.

For any integer $q \geq 1$, let $\mathcal{P}(\rset^q)$ denote the set of
probability measures on $\mathcal{B}(\rset^q)$.
For $\mu\in\mathcal{P}(\rset^d)$ and for a non-negative measurable
$f:\rset^d\to\rset$, we denote $\mu(f):=\int_{\rset^d} f(\theta)\mu(\rmd \theta)$.

For $\mu,\nu\in\mathcal{P}(\rset^d)$,
let $\mathcal{C}(\mu,\nu)$ denote the set of probability measures $\zeta$
on $\mathcal{B}(\rset^{2d})$ such that its respective marginals are $\mu,\nu$. Define, for $p\geq 1$,
\begin{equation}
\label{eq:definition-W-1}
W_p(\mu,\nu):=\left(
\inf_{\zeta\in\mathcal{C}(\mu,\nu)}\int_{\rset^d}\int_{\rset^d}|\theta-\theta'|^p\zeta(\rmd \theta \rmd \theta')\right)^{1/p},
\end{equation}
which is the $L^p$-Wasserstein distance associated to the Euclidean distance. We consider below only the cases $p=1,2$.

\section{Main results}


Fix an $\rset^d$-valued random variable $\theta_0$, representing the initial value of
the procedure we consider.
Let $(\mathcal{G}_n)_{n \in \nset}$ be a given filtration representing the flow of past information.
The notation $\mathcal{G}_{\infty}$ is self-explanatory.
Let $(X_n)_{n \in \nset}$ be a $(\mathcal{G}_n)$-adapted process.
Let furthermore $(\mathcal{G}^{+}_n)_{n \in \nset}$  be a decreasing sequence of $\sigma$-fields which represent
the future information at the respective time instants. We assume in the sequel that for each $n \in \nset$,
the $\sigma$-fields $\mathcal{G}_n$ and $\mathcal{G}_n^+$ are independent.

Fix $\beta>0$. For each $\lambda>0$, define the $\rset^d$-valued
random process $(\theta^{\lambda}_n)_{n \in \nset}$ by recursion:
\begin{equation}\label{nab}
\theta^{\lambda}_0:=\theta_0,\quad \theta^{\lambda}_{n+1}:=\theta^{\lambda}_n-\lambda H(\theta^{\lambda}_n,X_{n})+\{ 2 \lambda \beta^{-1} \}^{1/2} \, \xi_{n+1},\ n\in\nset,
\end{equation}
where $H:\rset^d\times\rset^m\to\rset^d$ is a measurable
function and $(\xi_n)_{n\in\nset}$ is an independent sequence of standard $d$-dimensional Gaussian random variables.

We interpret $(X_n)_{n\in\nset}$ as a stream of data and $(\xi_n)_{n\in\nset}$ as an artificially
generated noise sequence.
We assume throughout the paper that $\theta_0$, $\mathcal{G}_{\infty}$ and $(\xi_{n})_{n\in\nset}$
are independent.

Let $U:\rset^d\to\rset_+$ be continuously differentiable
with gradient $h:=\nabla U$.
Let us define the probability
\[
\pi_{\beta}(A):=\frac{\int_A \rme^{-\beta U(\theta)}\, \rmd \theta}{\int_{\rset^d} \rme^{-\beta U(\theta)}\, \rmd \theta},\
A\in\mathcal{B}(\rset^d).
\]
It is implicitly assumed that $\int_{\rset^d} \rme^{-\beta U(\theta)}\, \rmd \theta<\infty$ and this is
indeed the case under \Cref{assum:dissipativity} below, as easily seen.
Our objective is to (approximately) sample from the distribution $\pi_{\beta}$ using the
scheme \eqref{nab}.

We now present our assumptions. First, the moments of the initial condition need to
be controlled.

\begin{assumption}\label{imit}
$ |\theta_0|\in\bigcap_{p\geq 1} L^p$.
\end{assumption}

Next, we require joint Lipschitz-continuity of every coordinate function $H^{i}$, $i=1,\ldots,d$.

\begin{assumption}\label{assum:lip} There exist positive constants $K^{i}_1, K^{i}_2$, $i=1,\ldots,d$ such that for all
$\theta,\theta'\in\rset^d$ and  $x,x'\in\rset^m$,
\[
|H^{i}(\theta,x)-H^{i}(\theta',x')|\leq K^{i}_1|\theta-\theta'|+ K^{i}_2|x-x'|.\
\]
\end{assumption}

We set
\begin{equation}
\label{eq:definition-H-*}
H^*:=|H(0,0)|,\quad K_{1}:=\sum_{i=1}^{d}K^{i}_{1},\quad K_{2}:=\sum_{i=1}^{d}K^{i}_{2}
\end{equation}
and notice that, clearly,
\begin{equation}\label{liipa}
|H(\theta,x)-H(\theta',x')|\leq K_1|\theta-\theta'|+ K_2|x-x'|.
\end{equation}

\begin{remark}{\rm The reader may wonder why we did not assume just \eqref{liipa} directly for some $K_{1},K_{2}$. The reason is that our
estimates in the proof of Lemma \ref{lem:h_minus_cond_h} below lead to constants depending on $K_{1},K_{2}$
as defined by the sums of the respective
Lipschitz-constants for the coordinate mappings.}
\end{remark}

The data sequence $(X_n)_{n\in\nset}$ need not be i.i.d., we require
only a mixing property, defined in Section~\ref{lm} below.

\begin{assumption}\label{assum:lmiu}
Let $\mathcal{G}_n$, $n\in\mathbb{N}$ be a given filtration with $\mathcal{G}_{0}=\{\emptyset,\Omega\}$.
Let $\mathcal{G}_{n}^{+}$, $n\in\mathbb{N}$ be a decreasing family of sigma-algebras such that
$\mathcal{G}_{n}$ is independent of $\mathcal{G}_{n}^{+}$ for all $n\in\mathbb{N}$.
The process $(X_n)_{n\in\nset}$ is conditionally $L$-mixing with respect to
$(\mathcal{G}_n,\mathcal{G}^+_n)_{n\in\nset}$. It satisfies for each $\theta \in \rset^d$ and $n\geq 1$,
\begin{equation}\label{lopp}
\E[H(\theta,X_n)]=h(\theta) \, .
\end{equation}
\end{assumption}

If the process $(X_n)_{n\geq 1}$ happens to be strictly stationary then \eqref{lopp} clearly holds.
Finally, we present a dissipativity condition on $H$.
\begin{assumption}\label{assum:dissipativity}
There exist $a$, $b>0$ such that, for all $\theta \in \rset^d$ and $x\in\rset^m$,
\begin{equation}\label{principal}
\ps{H(\theta,x)}{\theta} \geq a |\theta|^2-b.
\end{equation}
\end{assumption}
When $X_n=c$ for all $n\in\nset$ for some $c\in\rset^m$
(i.e. when $H(\theta,X_{n+1})$ is replaced by $h(\theta)$
in \eqref{nab}) then we arrive at the well-known unadjusted Langevin algorithm
whose convergence properties have been amply analyzed, see e.g.
~\cite{dalalyan:2017,durmus:moulines:2017,durmus:moulines:highdimULA,cheng2018sharp,majka2018non} and the references therein. The case of i.i.d.
$(X_n)_{n\in\nset}$ has also been
investigated in great detail, see e.g.~\cite{raginsky2017non,xu2018global,majka2018non}.

In the present article, better estimates are obtained for
the distance between $\mathcal{L}(\theta^{\lambda}_n)$ and $\pi_{\beta}$
than those of \cite{raginsky2017non} and
\cite{xu2018global}. Such rates have already been obtained in \cite{barkhagen2018stochastic}
for strongly convex $U$ and in \cite{majka2018non} for $U$ that satisfies a monotonicity condition outside
a compact set. Here we make no convexity assumptions at all.
This comes at the price of using the metric $W_1$ defined in \eqref{eq:definition-W-1} below while
\cite{raginsky2017non,xu2018global,majka2018non,barkhagen2018stochastic} use Wasserstein distances with respect to the
standard Euclidean metric, see \eqref{eq:definition-W-1} below.

Another novelty of our paper is that, just like in \cite{barkhagen2018stochastic},
we allow the data sample $(X_n)_{n\in\nset}$ to be dependent. As observed data have no reason
to be i.i.d., we believe that such a result is
fundamental to assure the robustness of online optimization procedures based on the SGLD \eqref{nab}.

\begin{remark}
{\rm In this work, the constants appearing are often denoted by $C_j$ for some natural number
$j\in\nset$.
Without further mention, these constants depend on $K_1$, $K_2$, $a$, $b$,
$H^*$, $\beta$, $d$, and from the process $(X_n)_{n\in\nset}$ (such as its moments). Unless otherwise stated,
they do not depend on anything else. In case of further dependencies (e.g.\ in Lemma~\ref{lyapp} dependence on the order of the moment
$p$ appears),
we indicate these in parentheses, e.g.\ $C_6(p)$.}
\end{remark}

Our main contribution is summarized in the following result. Set
\begin{equation}
\label{eq:definition-lambda-max}
\lambda_{\max}= \min\{a/2K_1^2,1/a\} \,,
\end{equation}
where $K_1$ and $a$ are defined in \Cref{assum:lip} and \Cref{assum:dissipativity}, respectively.
\begin{theorem}\label{main} Assume \Cref{imit}, \Cref{assum:lip}, \Cref{assum:lmiu} and \Cref{assum:dissipativity}.
Then there are positive constants
$C_0$, $C_1$, $C_2$ such that, for  every $0<\lambda\leq \lambda_{\max}$, $\beta >0$ and $n \in \nset$,
\begin{equation}\label{manyi}
W_1(\mathcal{L}(\theta^{\lambda}_n),\pi_{\beta})\leq C_1 \rme^{-C_0\lambda n}\E[|\theta_{0}|^{4}+1]+C_2\sqrt{\lambda} \,,
\end{equation}
where $W_1$ is defined in \eqref{eq:definition-W-1}.
\end{theorem}
\begin{remark}{\rm Our assumptions can be somewhat weakened, as seen from a careful reading of the
proofs. Indeed, the above theorem remains valid if we assume (instead of conditional $L$-mixing) only
that $X_{n}$, $n\in\mathbb{N}$ are $L^{4}$-bounded and, for some $\epsilon>0$, the sequences
$M^n_{2+\epsilon}(X), \Gamma^n_{2+\epsilon}(X), n \in \nset$ are bounded in $L^{2}$
for some $\epsilon>0$.
Furthermore, Assumption \ref{imit} can be weakened to $|\theta_0|\in L^{6}$. }
\end{remark}

\cite[Example~3.4]{barkhagen2018stochastic} suggests that the best rate we can hope to get in \eqref{manyi} is $\sqrt{\lambda}$, even in the convex case.
The above theorem achieves this rate.
We remark that, although
the statement of Theorem \ref{main} concerns the discrete-time recursive
scheme \eqref{nab}, its proof is carried out entirely in a continuous-time setting, in Section \ref{po}. It relies
on techniques from \cite{barkhagen2018stochastic} and \cite{eberle:guillin:zimmer:Trans:2019}. The principal new ideas are
the introduction of the auxiliary process $\tilde{Y}^{\lambda}_t(\mathbf{x})$, $t\in\rset_+$
(see \eqref{mah} below) and reliance on the contractivity of the continuous system dynamics in a suitable semimetric (see Proposition
\ref{prop:contra} below).

\subsection{Related work and our contributions}
In \cite{raginsky2017non}, a non-convex empirical risk minimization problem is considered. The excess risk is decomposed into a sampling error resulting from the application of Stochastic Gradient Langevin Dynamic (SGLD), a generalization error and a suboptimality error. Our aim is to improve the sampling error in the non-convex setting and provide sharper convergence estimates under more relaxed conditions. To this end, we focus on the comparison of our results with \cite[Proposition~3.3]{raginsky2017non}.

\cite[Assumption~(A.5)]{raginsky2017non} is (much) stronger than \Cref{imit}
above. \Cref{assum:dissipativity} is identical to \cite[Assumption~(A3)]{raginsky2017non}.
\cite[Assumption~(A.2)]{raginsky2017non} corresponds to Lipschitz-continuity of $H$
in its first variable with a Lipschitz-constant 
independent from its second variable
and $(A.1)$ there means that $H(0,\cdot)$, $u(0,\cdot)$ are bounded
where $U(\theta)=\E[u(\theta,X_0)]$ and $H(\cdot,\cdot)=\partial_{\theta}u(\cdot,\cdot)$.
Hence \Cref{assum:lip} here is neither stronger nor weaker than $(A.2)$ of \cite{raginsky2017non},
they are incomparable conditions.
In any case, \Cref{assum:lip} does not seem to be restrictive.
Condition $(A.4)$ in \cite{raginsky2017non} is implied by  \Cref{assum:lip} and \Cref{assum:lmiu}.



We obtain stronger rates (which we believe to be optimal) than those of \cite{raginsky2017non}. More precisely, we obtain a rate $\lambda^{1/2}$ in \eqref{manyi} for the $W_1$
distance while  \cite{raginsky2017non} only obtains $\lambda^{5/4}n$ (which depends on $n$)
but in the ${W}_2$ distance.
Furthermore, \cite{raginsky2017non} is applicable only if $(X_n)_{n \in \nset}$ is i.i.d. while \Cref{assum:lmiu} suffices for the derivation of our results.


Now let us turn to \cite{majka2018non}. That paper assumes a strengthening of our dissipativity assumption:
they require \Cref{assum:lip} and that there exist $b,a>0$ such that, for each $\theta, \theta' \in \rset^d$ satisfying
$|\theta-\theta'|>b$,
\begin{equation}\label{condition}
\ps{h(\theta)-h(\theta')}{\theta-\theta'} \geq a |\theta-\theta'|^2,\ x\in\rset^m.
\end{equation}
Note, however, that this is stipulated only for $h$ in \cite{majka2018non}
while we need our dissipativity assumption for $H(\cdot,x)$, for all $x$, as we allow dependent data streams.
Furthermore, Assumption 1.3 in \cite{majka2018non} requires that the variance of
$H(\theta,X_0)$ is  controlled by a power of the step size $\lambda$ while we do not need
such an assumption. The second conclusion of their Theorem 1.4
(with $\alpha=1$, using their notation $\alpha$) is the same as that of our Theorem \ref{main}.




\section{Proofs}\label{po}

\subsection{Conditional $L$-mixing} \label{lm}

A key mixing assumption is required about $X_{n}$, $n\in\mathbb{N}$. In this
subsection we present some related concepts and results. The material presented here is from \cite{barkhagen2018stochastic}.

$L$-mixing processes and random fields were introduced in \cite{gerencser:1989}. In
\cite{chau:kumar:rasonyi:sabanis:2019}, the closely related concept of \emph{conditional} $L$-mixing
was created.

We assume that the probability space $(\Omega, \mathcal{F}, \P)$ is equipped
with a discrete-time filtration $(\mathcal{R}_n)_{n\in\nset}$ as well as with a decreasing sequence of sigma-fields $(\mathcal{R}_n^+)_{n\in\nset}$ such that the $\sigma$-fields $\mathcal{R}_n$ and $\mathcal{R}_n^+$ are independent for all $n \in \nset$.
{}
A random process $(U_n)_{n\in\nset}$ is called $L^r$-\emph{bounded} for some $r\geq 1$
if
\[
\sup_{n\in\nset}\E^{1/r}[|U_n|^r]<\infty.
\]
Define, for each $n\in\nset$, $i=1,\ldots,d$,
\begin{align}\nonumber
	\tilde{M}^{n}_r(U,i) &:= \sup_{m \in\nset}
	\CPE[1/r]{|U_{n+m}^{i}|^r}{\mathcal{F}_n},\\
	\nonumber\tilde{\gamma}^{n}_r(\tau,U,i)&:= \sup_{m\geq\tau}
	\CPE[1/r]{|U^{i}_{n+m}-\CPE{U^{i}_{n+m}}{\mathcal{F}_{n+m-\tau}^+\vee \mathcal{F}_n}|^r}{
	\mathcal{F}_n},\ \tau\geq 0,
\end{align}
where $U_{n+m}^{i}$ refers to the $i$th coordinate of $U_{n+m}$ in the above expressions. Finally, set
\begin{equation}
\label{eq:definition-Gamma}
\text{$\tilde{\Gamma}^{n}_r(U,i) := \sum_{\tau= 0}^{\infty}\tilde{\gamma}^{n}_r(\tau,U,i)$,
${M}^{n}_{r}(U) :=\sum_{i=1}^{k}\tilde{M}^{n}_{r}(U,i)$, and $\Gamma^{n}_{r}(U) :=\sum_{i=1}^{k}\tilde{\Gamma}^{n}_{r}(U,i)$.}
\end{equation}
\begin{definition}[Conditional $L$-mixing] We say that the random process
$(U_n)_{n\in\nset}$
is \emph{{conditionally} $L$-mixing}
with respect to $(\mathcal{R}_n,\mathcal{R}_n^+)_{n\in \nset}$ if
$(U_n)_{n\in\nset}$ is adapted to
$(\mathcal{R}_n)_{n\in\nset}$
for all $\theta\in \Param$;
for all $r\geq 1$,
it is $L^r$-bounded;
and the sequences  $(M^n_r(U))_{n\in \nset}$, $(\Gamma^n_r(U))_{n\in\nset}$
are also $L^r$-bounded for all $r\geq 1$.
\end{definition}

Conditionally $L$-mixing encompasses a broad class of stochastic models (i.i.d.\ with finite moments
of all orders, linear processes,
functionals of Markov processes, etc.), see  in \cite[Example 2.1]{barkhagen2018stochastic}.


{}
It is convenient to extend the $L$-mixing property to the continuous-time setting.
We consider a continuous-time filtration $(\mathcal{R}_t)_{t\in\mathbb{R}_+}$ as well as a decreasing family of sigma-fields
$(\mathcal{R}_t^+)_{t\in\mathbb{R}_+}$. We assume that $\mathcal{R}_t$ is
independent of $\mathcal{R}_t^+$, for all $t\in\mathbb{R}_+$.
Consider an $\mathbb{R}^{d}$-valued continuous-time stochastic process $(W_{t})_{t\in\mathbb{R}_+}$
which is progressively measurable (i.e.\ $W:[0,t]\times\Omega\to\mathbb{R}^d$ is $\mathcal{B}([0,t])\otimes\mathcal{R}_t$-measurable
for all $t\in\mathbb{R}_+$).
From now on we assume that $W_{t}\in L^{1}$, $t\in\mathbb{R}_{+}$.
We define the quantities\footnote{For a family $(Z_i)_{i\in I}$ of real-valued random variables
(where the index set $I$ may have arbitrary cardinality), there exists one and (up to
a.s.\ equality) only one random variable $g = \mathrm{ess}\sup_{i\in I} Z_i$ such that it dominates almost
surely all the $Z_i$ and it is a.s.\ dominated by any other random variable with this property.
For an existence proof, see e.g.\  \cite[Proposition~VI.1.1]{neveu}.}
\begin{align*}
	\tilde{M}_r^i(\mathbf{W}) &:= \mathrm{ess.}\sup_{t \in\mathbb{R}_+} \CPE[1/r]{|W_{t}^i|^{r}}{\mathcal{R}_{0}},\\
	\tilde{\gamma}^i_r(\tau,\mathbf{W}) &:=  \mathrm{ess.}\sup_{t\geq\tau}
	\CPE[1/r]{|W_{t}^i- \CPE{W_{t}^i}{{\mathcal{R}_{t-\tau}^+\vee \mathcal{R}_0}}|^r}{\mathcal{R}_{0}},\ \tau\in\mathbb{R}_+,
\end{align*}
and set
\begin{equation*}
\text{
$M_r(\bW) := \sum_{i=1}^d \tilde{M}_r^i(\bW)$, $\tilde{\Gamma}_r^i(\mathbf{W}) := \sum_{\tau=0}^{\infty} \tilde{\gamma}^i_r(\tau,\bW)$, and $\Gamma_r(\mathbf{W}) := \sum_{i=1}^d \tilde{\Gamma}_r^i(\bW)$}
\end{equation*}
where $W_t^i$ refers to the $i$th coordinate of $W_t$.
We recall \cite[Theorem~B.5]{barkhagen2018stochastic} which is key to further developments.
\begin{theorem}\label{estim} Let $(W_{t})_{t \in \rset_+}$ be $L^{r}$-bounded for some $r> 2$
and let $M_{r}(\bW)+\Gamma_{r}(\bW)<\infty$ a.s.
Assume $\CPE{W_{t}}{\mathcal{R}_{0}}=0$ a.s.\ for $t\in\mathbb{R}_+$.
Let $f:[0,T]\to\mathbb{R}$ be $\mathcal{B}([0,T])$-measurable with $\int_{0}^{T}f_{t}^{2}\, \rmd t<\infty$.
Then there is a constant $C'(r)$ such that
\begin{equation}\label{erd}
\textstyle{\CPE[1/r]{\sup_{s\in [0,T]}\left|\int_{0}^{s} f_t W_t\, \rmd t \right|^r}{\mathcal{R}_{0}}
\leq C'(r)\left( \int_{0}^{T} f_t^{2}\, \rmd t \right)^{1/2} [M_r(\bW) +\Gamma_r(\bW)], \text{a.s.} }
\end{equation}
We can actually take
\[
\textstyle{C'(r)=\frac{\sqrt{r-1}}{2^{1/2}-2^{1/r}}.}
\]
\end{theorem}

Estimates for $M_{r}(\mathbf{W}),\Gamma_{r}(\mathbf{W})$ imply similar estimates for
functionals of $\mathbf{W}$.
\begin{lemma}\label{lem:below}
Assume \Cref{assum:lip}. Then, for each $i \in \nset$ and $\theta\in \boule{i}$, $(H(\theta,W_t))_{t\in\mathbb{R}_{+}}$ satisfies
\begin{equation}\label{mamma}
M_r(H(\theta,\mathbf{W}))\leq K_1 i + K_2 M_r(\mathbf{W}) + H^*,   
\end{equation}
where $H^*$ is defined in \eqref{eq:definition-H-*} and
\begin{equation}\label{gomma}
\Gamma_r(H(\theta,\mathbf{W}))\leq 2K_2\Gamma_r(\mathbf{W}).
\end{equation}
\end{lemma}
\begin{proof} Identical to the proofs in \cite[Lemma~6.4 and Example~2.4]{barkhagen2018stochastic}, using
Lipschitz-continuity of the coordinate functions $H^{i}$ with the respective constants $K_{1}^{i}$, $K_{2}^{i}$.
\end{proof}

One of the main advantages of the mixing concepts we use is that one can plug in $\mathcal{R}_{0}$-measurable random
variables into $\theta$ and still preserve the mixing properties.
\begin{lemma}\label{lemma_63}
Assume \Cref{assum:lip} and set $i \in \nset$. let $(Z_s)_{s \geq 0}$ be a family of $\boule{i}$-valued random variables satisfying $Z:\mathbb{R}_{+}\times\Omega\to\mathbb{R}^{d}$ is $\mathcal{B}(\mathbb{R}_{+})\otimes\mathcal{R}_{0}$-measurable.
Define the process $Y_t = H(Z_{t},W_{t})$ for  $t\in\mathbb{R}_{+}$. Then
\[
M_p(\mathbf{Y}) \le K_1 i + K_2 M_r(\mathbf{W}) + H^*,   
\]
and
\[{}
\Gamma_r(\mathbf{Y})\leq 2K_2\Gamma_r(\mathbf{W}).
\]
\end{lemma}
\begin{proof}
The proof is identical to that of \cite[Lemma~A.3]{chau:kumar:rasonyi:sabanis:2019}, noting the Lipschitz continuity.
\end{proof}

\subsection{Further notations and introduction of auxiliary processses}

Note that \Cref{assum:lip} implies
\begin{equation}\label{mulyan}
|h(\theta)-h(\theta')|\leq K_1|\theta-\theta'|,\ \theta,\theta'\in\rset^d,
\end{equation}
and Assumption \Cref{assum:dissipativity} implies
\begin{equation}\label{eq:laban}
\ps{h(\theta)}{\theta} \geq a |\theta|^2-b,\ \theta\in\rset^d.
\end{equation}
Also, \Cref{assum:lip} implies
\begin{equation}\label{jojo}
|H(\theta,x)|\leq K_1|\theta|+K_2|x|+H^*,
\end{equation}
with the constant $H^*$ defined in \eqref{eq:definition-H-*}. Define, for each $p\geq 2$,
\begin{equation}
\label{eq:definition-Vp}
V_p(\theta)= \operatorname{v}_p(|\theta|) \, , \quad \text{where for $u \in \rset_+$, $\operatorname{v}_p (u):= (1+u^2)^{p/2}$}.
\end{equation}
We use $V_{p}$  as a Lyapunov function which allows to obtain uniform bounds for the moments of various processes.
Notice that each $V_{p}$ is twice continuously differentiable and
\begin{equation}\label{solymos}
\lim_{|\theta|\to\infty}\frac{\nabla V_p(\theta)}{V_p(\theta)}=0.
\end{equation}

{}
Let $\mathcal{P}_{\, V_p}(\rset^d)$ denote the subset of $\mu\in\mathcal{P}(\rset^d)$ satisfying
$\int_{\rset^d}V_p(\theta)\,\mu(\rmd \theta)<\infty$.
The following functional is pivotal in our arguments as it is used to measure the distance between probability measures.
We define, for any $p\geq 1$ and
$\mu,\nu \in \mathcal{P}_{\, V_p}(\rset^d)$,
\begin{equation}
\label{eq:definition-w-1-p}
w_{1,p}(\mu,\nu):=\inf_{\zeta\in\mathcal{C}(\mu,\nu)}\int_{\rset^d}\int_{\rset^d} \{1\wedge |\theta-\theta'|\} \{1+V_p(\theta)+V_p(\theta') \}\zeta(\rmd \theta \rmd \theta'),
\end{equation}
Though $w_{1,p}$ is not a metric, it satisfies
\begin{equation} \label{eq:lucia}
W_1(\mu,\nu)\leq w_{1,p}(\mu,\nu),
\end{equation}
as easily seen, where $W_1$ is defined in \eqref{eq:definition-W-1}.
In the sequel we solely consider the case $p=2$, that is, $w_{1,2}$.

Our estimations are carried out below in a \emph{continuous-time} setting, so
we define and discuss a number of auxiliary {continuous-time} processes below.
First, consider $(L_t)_{t\in\rset_+}$ defined by
the stochastic differential equation (SDE)
\begin{equation}\label{kako}
\rmd L_t=-h(L_t)\, \rmd t+ \{ 2 \beta^{-1} \}^{1/2}\, \rmd B_t,\quad L_0:=\theta_0,
\end{equation}
where $(B_t)_{t \geq 0}$ is standard Brownian motion on $(\Omega,\mathcal{F},\P)$,
independent of $\mathcal{G}_{\infty}\vee \sigma(\theta_0)$ with its natural filtration
denoted by $(\mathcal{F}_t)_{t\in\rset_+}$ henceforth. The meaning of $\mathcal{F}_{\infty}$
is clear.

Equation \eqref{kako}
has a unique solution on $\rset_+$ adapted to $(\mathcal{F}_t)_{t\in\rset_+}$ since $h$ is
Lipschitz-continuous by \eqref{mulyan}.
We proceed by defining, for each $\lambda>0$ convenient time-changed versions of $L_{t}$, $t\in\mathbb{R}_{+}$:
\[
L^{\lambda}_t:=L_{\lambda t},\ t\in\rset_+.
\]
Notice that $\tilde{B}^{\lambda}_t:=B_{\lambda t}/\sqrt{\lambda}$, $t\in\rset_+$
is also a Brownian motion and
\begin{equation}\label{eq:kell}
\rmd L^{\lambda}_t=-\lambda h(L^{\lambda}_t)\, \rmd t+\{ 2 \lambda \beta^{-1} \}^{1/2} \, \rmd \tilde{B}^{\lambda}_t,\
L^{\lambda}_0=\theta_0.
\end{equation}
Define $\mathcal{F}_t^{\lambda}:=\mathcal{F}_{\lambda t}$, $\lambda \in \rset_+$, $t\in\rset_+$, the natural
filtration of $(\tilde{B}^{\lambda}_t)_{t \geq 0}$.

Our recursion \eqref{nab} is defined in terms of the data sequence $X_{n}$, $n\in\mathbb{N}$. However, it is more convenient
to \emph{freeze} the values of this sequence and to do the analysis initially with such framework. To this end,  for each $\lambda>0$ and $\mathbf{x}=(x_0,x_1,\ldots)\in(\rset^m)^{\nset}$, consider the process $(\tilde{Y}_t^{\lambda}(\mathbf{x}))_{t\in\rset_+}$ defined as
\begin{equation}\label{mah}
\rmd \tilde{Y}^{\lambda}_t(\mathbf{x})=-\lambda H(\tilde{Y}^{\lambda}_t(\mathbf{x}),
x_{\lfloor t\rfloor})\, \rmd t+\{ 2 \lambda \beta^{-1} \}^{1/2} \, \rmd \tilde{B}^{\lambda}_{t},
\end{equation}
with initial condition $\tilde{Y}^{\lambda}_0(\mathbf{x})=\theta_0$.
Due to \Cref{assum:lip}, there is a unique
solution to \eqref{mah} which is adapted to
$(\mathcal{F}_t^{\lambda})_{t\in\rset_+}$. This process provides a continuous-time
``approximation'' for our recursive procedures and plays an important role in the estimations below.

Moreover, for any given $s\ge 0$ and $t\ge s$, consider the following auxiliary process, which follows
the same dynamics as \eqref{mah} but  its starting time and value are prescribed:
\begin{equation}\label{eq:aux_proc_conts}
\rmd \tzeta{\lambda}{t}{s}{\mathbf{x}}{\theta} = -\lambda H(\tzeta{\lambda}{t}{s}{\mathbf{x}}{\theta}, x_{\lfloor t\rfloor}) \rmd t + \{ 2 \lambda \beta^{-1} \}^{1/2} \, \rmd \tilde{B}_t^{\lambda}, \qquad \mbox{for }t> s,
\end{equation}
with initial condition $\tzeta{\lambda}{s}{s}{\mathbf{x}}{\theta}=\theta\in\mathbb{R}^{d}$.{}
Note that $\tzeta{\lambda}{t}{s}{\mathbf{x}}{\tilde{Y}^{\lambda}_s(\mathbf{x})}=\tilde{Y}^{\lambda}_t(\mathbf{x})$ for all $t > s$ and
for all $\mathbf{x}=(x_0,x_1,\ldots)\in(\rset^m)^{\nset}$.

Let us now define the continuously interpolated
Euler-Maruyama approximation of $(\tilde{Y}^{\lambda}_t(\mathbf{x}))_{t\in\rset_+}$ via
\begin{equation}\label{mahh}
\rmd Y^{\lambda}_t(\mathbf{x})=-\lambda H(Y^{\lambda}_{\lfloor t\rfloor}(\mathbf{x}),{x}_{\lfloor t\rfloor})\, \rmd t
+ \{ 2 \lambda \beta^{-1} \}^{1/2} \, \rmd \tilde{B}^{\lambda}_{t},
\end{equation}
with initial condition $Y^{\lambda}_0(\mathbf{x})=\theta_0$.
Notice at this point that \eqref{mahh} can be solved by a simple recursion.

Now we explain the relationship of the latter process to $\theta^{\lambda}_{n}$, $n\in\mathbb{N}$, defined in \eqref{nab}.
If one considers $(Y^{\lambda}_t(\mathbf{X}))_{t\in\rset_+}$,
where $\mathbf{X}=(X_0,X_1,\dots)$ is a random element in $(\rset^m)^{\nset}$, then for each integer $n\in\nset$,
\begin{equation}\label{antoniojobim}
\mathcal{L}(Y^{\lambda}_n(\mathbf{X}))=\mathcal{L}(\theta_n^{\lambda}),
\end{equation}
since $\tilde{B}^{\lambda}_{n+1}-\tilde{B}^{\lambda}_{n}$ has standard Gaussian law on $\mathbb{R}^{d}$, for all $n\in\mathbb{N}$.

\subsection{Layout of the proof}

In view of the observation \eqref{antoniojobim}, the
main objective is to bound $W_1(\mathcal{L}(Y^{\lambda}_t(\mathbf{X})),\pi_{\beta})$. This task can be decomposed
as follows:
\begin{equation}
 W_1(\mathcal{L}(Y^{\lambda}_t(\mathbf{X})),\pi_{\beta}) \\
\le W_1(\mathcal{L}(Y^{\lambda}_t(\mathbf{X})),\mathcal{L}(\tilde{Y}^{\lambda}_t(\mathbf{X}))) +
W_1(\mathcal{L}(\tilde{Y}^{\lambda}_t(\mathbf{X})),\mathcal{L}(L^{\lambda}_t)) + W_1(\mathcal{L}(L^{\lambda}_t),\pi_{\beta}) \,,
\label{alambda}
\end{equation}
where $\pi_\beta$ is defined in \eqref{eq:definition-pi-beta}.
Here the last term is controlled below by standard arguments which entail that $L_{t}^{\lambda}$ converges in law to $\pi_{\beta}$
as $t\to\infty$. The drift condition \eqref{eq:genlib} below (which follows from the dissipativity
Assumption \Cref{assum:dissipativity} and Lipschitzness of the mean field $h$, see \eqref{mulyan}) ensure the applicability
of classical results.

The second term is controlled uniformly in $t$ by a quantity which is proportional to $\sqrt{\lambda}$.
To this end, we follow the line of attack used in \cite{barkhagen2018stochastic} which consists in estimating,
on intervals of length $1/\lambda$, the $L^{2}$-distance between $\tilde{Y}_{t}^{\lambda}(\mathbf{X})$ and another process
that coincides with it at the initial point of the interval but follows the averaged dynamics \eqref{eq:kell} (see \eqref{eq:definition-overline-Z} for a precise definition and Lemma~\ref{vizier} for details).
Here we rely on a maximal inequality for functionals of a conditionally $L$-mixing process, given as Theorem \ref{estim} above.
We put together estimates on separate intervals and thus obtain a bound on
$W_1(\mathcal{L}(\tilde{Y}^{\lambda}_t(\mathbf{X})),\mathcal{L}(L^{\lambda}_t))$ in Lemma \ref{intermediate}, relying on novel
results by \cite{eberle:guillin:zimmer:Trans:2019}, which give us a contraction rate for the diffusion
$L^{\lambda}_{t}$, $t\geq 0$ in the semimetric $w_{1,2}$,
see Proposition \ref{prop:contra} and, in particular, \eqref{karako}.

Finally, the first term is controlled uniformly in $t$ by a quantity which is also proportional to $\sqrt{\lambda}$, see Corollary~\ref{crux}.
This is based on Kullback-Leibler distance estimates which go back to \cite{dalalyan:tsybakov:2012} but which are
somewhat trickier as we need to employ measurable selection to pass from bounds for
$W_1(\mathcal{L}(Y^{\lambda}_t(\mathbf{x})),\mathcal{L}(\tilde{Y}^{\lambda}_t(\mathbf{x})))$ with fixed $\mathbf{x}$
to ones for $W_1(\mathcal{L}(Y^{\lambda}_t(\mathbf{X})),\mathcal{L}(\tilde{Y}^{\lambda}_t(\mathbf{X})))$.

\subsection{Moment estimates}
Define the following notation for $\lambda > 0$, $\beta > 0$, $\theta \in \rset^d$, $x \in \rset^m$,
\begin{align}
\nonumber
\tgenerator &:=
\lambda \beta^{-1} \Delta V_p(\theta)-\lambda \ps{h(\theta)}{\nabla V_p(\theta)} \, , \\
\label{eq:definition-generator-L}
\generator[\lambda][\beta][x][p][\theta] &:=
\lambda \beta^{-1} \Delta V_p(\theta)-\lambda \ps{H(\theta,x)}{\nabla V_p(\theta)} \,.
\end{align}

\begin{lemma}\label{lyapp} Assume \Cref{assum:dissipativity}. For each $p\geq 2$, $\theta \in \rset^d$, and $x \in \rset^m$,
\begin{align}
\label{eq:genlib}
&\tgenerator[1][\beta][p][\theta] \leq -C_6(p) V_p(\theta)+C_7(p) \, \\
\label{eq:genlib1}
&\generator[1][\beta][x][p][\theta] \leq -C_6(p) V_p(\theta)+C_7(p),\ \theta\in\rset^d,
\end{align}
where $C_6(p) = ap/4$, $C_7(p) = (3/4)ap \mathrm{v}_p(\overline{M}(p))$ with
\begin{equation}
\label{eq:definition-barM-p}
\overline{M}(p) = \sqrt{1/3+4b/(3a)+4d/(3a\beta)+4(p-2)/(3a\beta)} \,.
\end{equation}
\end{lemma}
\begin{proof}
By direct calculation,
\begin{equation}\label{ross}
\tgenerator[1] = \beta^{-1} dp V_{p-2}(\theta) +\beta^{-1} p(p-2)(|\theta|^2+1)^{(p-4)/2}|\theta|^2 \\
- p V_{p-2}(\theta) \ps{h(\theta)}{\theta} \,.
\end{equation}
By \Cref{assum:dissipativity}, see also \eqref{eq:laban}, the third term of (\ref{ross}) is dominated by
\begin{equation}\label{tarkin}
-pa |\theta|^2(|\theta|^2+1)^{(p-2)/2}+pb(|\theta|^2+1)^{(p-2)/2}.
\end{equation}
Then,  for $|\theta| >\overline{M}(p)$, one observes that
$\tgenerator[1] \leq -(\fraca{ap}{4})V_p(\theta)$.
As for $|\theta| \leq \overline{M}(p)$, one obtains $\tgenerator[1] \leq (\fraca{3}{4})ap \mathrm{v}_p(\overline{M}(p))$.
Eq.~\eqref{eq:genlib} follows. The statement \eqref{eq:genlib1} follows in an identical way, noting that the constants which appear do not depend on $x \in \rset^m$.
\end{proof}
Now, we proceed with the required moment estimates which play a crucial
role in the derivation of the main result as given in Theorem \ref{main}.

\begin{lemma} \label{lem:aux_proc_conts_V_p}
Assume \Cref{imit}, \Cref{assum:lip} and \Cref{assum:dissipativity}. Let $p\geq 2$ and $\tilde{\theta} \in L^{2p-2}$. For any $t> s \ge 0$,
\begin{equation} \label{eq:aux_proc_conts_V_p}
\sup_{\mathbf{x}\in(\rset^m)^{\nset}} \E[V_p(\tzeta{\lambda}{t}{s}{\mathbf{x}}{\tilde{\theta}})] \le \rme^{-\lambda C_6(p)(t-s)} \E[V_{p}(\tilde{\theta})] + 3 \mathrm{v}_p(\overline{M}(p))
\end{equation}
where $\overline{M}(p)$ is defined in \eqref{eq:definition-barM-p}.
\end{lemma}

\begin{proof} We note that
$2p-2\geq p$ for $p \geq 2$, hence $\E[V_{p}(\tilde{\theta})]<\infty$.
For any fixed sequence $\mathbf{x} \in (\rset^m)^{\nset}$ and $t>s\geq 0$, by It\^o's formula, one obtains almost surely,
\begin{equation*}
\rmd V_p(\tzeta{\lambda}{t}{s}{\mathbf{x}}{\tilde{\theta}})=
\generator[\lambda][\beta][x_{\lfloor t \rfloor}][p][\tzeta{\lambda}{t}{s}{\mathbf{x}}{\tilde{\theta}}] \rmd t + \{2\lambda \beta^{-1}\}^{1/2} \, \ps{\nabla V_p(\tzeta{\lambda}{t}{s}{\mathbf{x}}{\tilde{\theta}})}{ \rmd \tilde{B}^{\lambda}_t}\\
 \,,
\end{equation*}
Since $\sup_{0\leq s\leq t} \E[ |\nabla V_p(\tzeta{\lambda}{t}{s}{\mathbf{x}}{\tilde{\theta}})|^2]<\infty$ using $\tilde{\theta}\in L^{2p-2}$,
the expectation of the stochastic integral vanishes and
\begin{equation*}
\E[V_p(\tzeta{\lambda}{t}{s}{\mathbf{x}}{\tilde{\theta}})] = \E[V_p(\tilde{\theta})]
+\int_s^t\E\left[\generator[\lambda][\beta][x_{\lfloor u \rfloor}][p][\tzeta{\lambda}{u}{s}{\mathbf{x}}{\tilde{\theta}}]\right]\, \rmd u,
\end{equation*}
Differentiating both sides and using Lemma~\ref{lyapp} yields that
\begin{equation} \label{derivative}
\frac{\rmd}{\rmd t}\E[V_p(\tzeta{\lambda}{t}{s}{\mathbf{x}}{\tilde{\theta}})]  	
 =  \E\left[\generator[\lambda][\beta][x_{\lfloor t \rfloor}][p][\tzeta{\lambda}{t}{s}{\mathbf{x}}{\tilde{\theta}}]\right] \leq  -\lambda C_6(p)\E[ V_p(\tzeta{\lambda}{t}{s}{\mathbf{x}}{\tilde{\theta}})] +\lambda C_7(p).
\end{equation}
Hence, by calculating the derivative of $\rme^{\lambda C_6(p) (t-s)} \E[V_p(\tzeta{\lambda}{t}{s}{\mathbf{x}}{\tilde{\theta}})] $ and in view of the above relationship \eqref{derivative}, one obtains \eqref{eq:aux_proc_conts_V_p}.
\end{proof}

\begin{corollary} \label{lem:moments}
Assume \Cref{imit}, \Cref{assum:lip} and \Cref{assum:dissipativity}. For any integer $p\geq 2$ and $t \in \rset_+$,
\begin{equation}\label{momenteq1}
\sup_{\mathbf{x}\in(\rset^m)^{\nset}} \E[V_p(\tilde{Y}^{\lambda}_t(\mathbf{x}))]
\leq \rme^{-\lambda C_6(p) t} \E[V_p(\theta_0)]+ 3\mathrm{v}_p(\overline{M}(p)),
\end{equation}
where  $\overline{M}(p)$ is defined in \eqref{eq:definition-barM-p}.
\end{corollary}
\begin{proof}
By noting that $\tilde{Y}^{\lambda}_t(\mathbf{x}) = \tzeta{\lambda}{t}{0}{\mathbf{x}}{\theta_0}$, one immediately recovers the desired result from Lemma~\ref{lem:aux_proc_conts_V_p}.
\end{proof}
\begin{corollary}\label{cor:SDE_SGLDsmoments}
Assume \Cref{imit}, \Cref{assum:lip} and \Cref{assum:dissipativity}. For any integer $p\geq 2$ and $t \in \rset_+$,
\begin{equation}\label{SDE_SGLDsmoments_q1}
\E[V_p(\tilde{Y}^{\lambda}_t(\mathbf{X}))]
\leq \rme^{-\lambda C_6(p)t} \E[V_p(\theta_0)]+ 3\mathrm{v}_p(\overline{M}(p)),
\end{equation}
where the constant $\overline{M}(p)$ is defined in \eqref{eq:definition-barM-p}.
\end{corollary}
\begin{proof}
Due to the fact that the dissipativity  condition \Cref{assum:dissipativity}  is uniform in $x$, all estimates are independent of $x$ and therefore the result follows immediately from Corollary~\ref{lem:moments}.
\end{proof}

While the moment estimates for $\tilde{Y}^{\lambda}(\mathbf{x})$ have been rather straightforward, similar bounds for
$Y^{\lambda}_{t}(\mathbf{x})$ require more substantial calculations, based again on dissipativity, see Assumption \Cref{assum:dissipativity}.

\begin{lemma}
\label{lem:moment_SGLD_2p}
Assume \Cref{imit}, \Cref{assum:lip} and \Cref{assum:dissipativity}. For any   $\lambda < \lambda_{\max}$, as given in \eqref{eq:definition-lambda-max}, $n\in\nset$, $t \in \ocint{n,n+1}$, $p \in \nset^*$, and any sequence $\mathbf{x} \in (\rset^m)^{\nset}$,
\begin{multline}\label{eq:moment_SGLD_2p}
\E[|Y^{\lambda}_t(\mathbf{x})|^{2p}] 	 \leq (1-a\lambda(t-n))(1-a\lambda)^n \E|\theta_0|^{2p}
\\+\lambda a M(p,d) \left\{|x_n|^{2p} + (1-a\lambda(t-n)) \sum\nolimits_{j=1}^{n} \left(1-a\lambda\right)^{j-1}|x_{n-j}|^{2p}  \right\} + \widehat{M}(p,d),
\end{multline}
where the constants $M(p,d)$ and $\widehat{M}(p,d)$ are given by
\begin{equation}\label{eq:def:M_and_Mhat}
     M(p,d)  = \left(2\lambda_{\max}+4/a\right)^{p-1} \left[1/a+d\tilde{M}^2(p)\right]  \cmoment^p
\end{equation}
and
\begin{multline*}
\widehat{M}(p,d)  = M(p,d) (c_2/c_0)^p + \tilde{M}^2(p)\left(\lambda_{\max}+2/a\right)^{p-1} \left(d+(1/\beta)^{p-1}\left(2dp(2p-1)\right)^p\right)
\end{multline*}
with
\begin{equation}\label{eq:def:M_tilde}
\tilde{M}(p) := 2^p\sqrt{p(2p-1)/(a\beta)}.
\end{equation}
and $\cmoment$ and $\cmomentone$ are defined by
\begin{equation}
\label{eq:definition-cmoment}
\cmoment = 8 K_2^2 \lambda_{\max}, \, \cmomentone= a^{-1}(c_2 +  2d\beta^{-1}) \quad  \text{and}  \quad c_2=2b+8\lambda_{\max} (H^*)^2 \, .
\end{equation}
In particular,
\begin{equation}\label{eq:moment_SGLD_2}
\E |Y^{\lambda}_{n+1}(\mathbf{x})|^2 \leq
 (1-a\lambda)^{n+1} \E|\theta_0|^2 + \lambda\cmoment \sum_{j=0}^{n} (1-a\lambda)^{j} | x_{n -j}|^2  + \cmomentone \, ,
\end{equation}
\end{lemma}
\begin{proof}
For any $n \in \nset$ and $t\in \ocint{n, n+1}$, define
$\tD[n][t][\bx] = Y^\lambda_n(\bx) - \lambda H(Y^{\lambda}_n(\bx), x_n)(t-n)$.
It is easily seen that for $t\in \ocint{n, n+1}$
\begin{align*}
\CPE{|Y^\lambda_t(\bx)|^{2}}{Y^\lambda_n(\bx)} 	
= |\tD[n][t][\bx]|^2 + (2 \lambda/\beta)d (t-n) .
\end{align*}
Using \Cref{assum:lip} and \Cref{assum:dissipativity}, one obtains for all $\lambda \leq \lambda_{\max}$,
\begin{align}
&|\tD[n][t][\bx]|^2	 = |Y_n^\lambda(\bx)|^2-2\lambda (t-n)
\ps{Y_n^\lambda(\bx)}{H(Y_n^\lambda(\bx), x_n)}+\lambda^2|H(Y_n^\lambda(\bx), x_n)(t-n)|^2 \nonumber \\
& \leq (1-2a\lambda(t-n))|Y_n^\lambda(\bx)|^2+2b\lambda(t-n)+2\lambda^2(t-n)^2\{K_1^2|Y_n^\lambda(\bx)|^2 +4K_2^2|x_n|^2+4(H^*)^2\} \nonumber \\
& \leq (1-a\lambda(t-n))|Y_n^\lambda(\bx)|^2+\lambda (t-n)(\cmoment|x_n|^2+ c_2). \label{long}
\end{align}
The desired result \eqref{eq:moment_SGLD_2} follows from an easy induction. For higher moments, the calculation is somewhat more involved. To this end, one calculates, by setting $\tU= \{ 2 \lambda \beta^{-1} \}^{1/2}(\tilde{B}_t^{\lambda}-\tilde{B}_n^{\lambda})$, for $t \in \coint{n,n+1}$,
\begin{multline*}
\E[|Y^{\lambda}_t(\mathbf{x})|^{2p}|Y^{\lambda}_n(\mathbf{x})]
\leq |\tD[n][t][\bx]|^{2p}+2p\E\left[\left.|\tD[n][t][\bx]|^{2p-2}
\ps{\tD[n][t][\bx]}{\tU}\right|Y^{\lambda}_n(\mathbf{x})\right] \\
+\sum_{k = 2}^{2p} \binom{2p}{k}
\E\left[\left. |\tD[n][t][\bx]|^{2p-k}\left|\tU\right|^{k}\right|Y^{\lambda}_n(\mathbf{x})\right], \end{multline*}
where the  last inequality is due to Lemma \ref{trivial}. The following inequality is used in the subsequent analysis
\begin{equation}
\label{eq:useful-bound}
(r+s)^p \leq (1+\epsilon)^{p-1} r^p + (1+\epsilon^{-1})^{p-1} s^p,
\end{equation}
where $p\ge 2$, $r,\,s \ge 0$ and $\epsilon>0$. We continue as follows
\begin{align*}
&\E[|Y^{\lambda}_t(\mathbf{x})|^{2p}|Y^{\lambda}_n(\mathbf{x})] \\
&\le |\tD[n][t][\bx]|^{2p}+ \sum_{l = 0}^{2(p-1)} \binom{2p}{l+2} \CPE{|\tD[n][t][\bx]|^{2(p-1)-l}|\tU |^{l}
|\tU|^2}{Y^{\lambda}_n(\mathbf{x})}  \nonumber \\
&\leq |\tD[n][t][\bx]|^{2p}+\binom{2p}{2} \sum_{l = 0}^{2(p-1)} \binom{2(p-1)}{l} \CPE{|\tD[n][t][\bx]|^{2(p-1)-l} |\tU|^l |\tU|^2 }{Y^{\lambda}_n(\mathbf{x})}  \nonumber \\
&= |\tD[n][t][\bx]|^{2p}+ p(2p-1)
\CPE{\left(|\tD[n][t][\bx]|+ |\tU|\right)^{2p-2}
|\tU|^2 }{Y^{\lambda}_n(\mathbf{x})}  \nonumber \\
& = |\tD[n][t][\bx]|^{2p} +\lambda(t-n)2^{2p-2}p(2p-1)d \beta^{-1}|\tD[n][t][\bx]|^{2p-2} +2^{2p-3}p(2p-1)\E\left[|\tU|^{2p}\right]
\end{align*}
which yields, using moment estimates  given in \cite[Theorem~7.1, Chapter~1]{mao:1997}, that
\begin{multline}\label{sgldbdap1}
\E[|Y^{\lambda}_t(\mathbf{x})|^{2p}|Y^{\lambda}_n(\mathbf{x})]
\leq |\tD[n][t][\bx]|^{2p} +\lambda(t-n)2^{2p-2}p(2p-1) d\beta^{-1} |\tD[n][t][\bx]|^{2p-2}  \\
+2^{3p-3}\left(\lambda(t-n)\right)^p(p(2p-1))^{p+1}\left\{d \beta^{-1} \right\}^p.
\end{multline}
Using \eqref{eq:definition-cmoment} and the inequalities \eqref{long} and \eqref{eq:useful-bound} with $\epsilon= a \lambda (t-n)/2$, one calculates
\begin{align}\label{sgldbdap2}
&|\tD[n][t][\bx]|^{2p} 	\leq \{(1-a\lambda(t-n))|Y^{\lambda}_n(\mathbf{x})|^2+\lambda  (t-n)(\cmoment|x_n|^2+c_2)\}^p \nonumber\\
					& \leq (1+\frac{a\lambda(t-n)}{2})^{p-1}(1-a\lambda(t-n))^p|Y^{\lambda}_n(\mathbf{x})|^{2p} +(1+\frac{2}{a\lambda(t-n)})^{p-1}\lambda^p (t-n)^p (\cmoment|x_n|^2+c_2)^{p} \nonumber\\
					& \leq
a_{n,t}^{\lambda,p} |Y^{\lambda}_n(\mathbf{x})|^{2p} + b_{n,t}^{\lambda,p} \,
\end{align}
where $a_{n,t}^{\lambda,p} = (1 -\fraca{a\lambda(t-n)}2)^{p-1}(1-a\lambda(t-n))$ and
$b_{n,t}^{\lambda,p}= (\lambda(t-n)+\fraca{2}{a})^{p-1}\lambda (t-n) (\cmoment|x_n|^2+c_2)^{p}$.
Substituting \eqref{sgldbdap2} into \eqref{sgldbdap1} yields
\begin{multline} \label{long_calc_mom}
\E[|Y_t^{\lambda}(\mathbf{x})|^{2p}|Y^{\lambda}_n(\mathbf{x})] 	\leq  a_{n,t}^{\lambda,p} |Y^{\lambda}_n(\mathbf{x})|^{2p} +b_{n,t}^{\lambda,p}
+\lambda(t-n)2^{2p-2}p(2p-1) d \beta^{-1} \\
\times \bigg[ a_{n,t}^{\lambda,p-1}|Y^{\lambda}_n(\mathbf{x})|^{2(p-1)} + b_{n,t}^{\lambda,p-1} \bigg] +2^{3p-3}\left(\lambda(t-n)\right)^p(p(2p-1))^{p+1}\left(d \beta^{-1} \right)^p.
\end{multline}
Define $\tilde{M}(p)$ as in \eqref{eq:def:M_tilde} and observe that for $|Y^{\lambda}_n(\mathbf{x})| \ge \sqrt{d}\tilde{M}(p)$
\[
\frac{a\lambda(t-n)}{4}|Y^{\lambda}_n(\mathbf{x})|^{2p} \ge \lambda(t-n)2^{2p}p(2p-1)\frac{d}{4\beta}|Y^{\lambda}_n(\mathbf{x})|^{2(p-1)}.
\]
Consequently, on $\{|Y^{\lambda}_n(\mathbf{x})| \ge \sqrt{d}\tilde{M}(p)\}$ the inequality \eqref{long_calc_mom} yields
\begin{align} \label{long_calc_mom2}
&\E[|Y_t^{\lambda}(\mathbf{x})|^{2p}|Y^{\lambda}_n(\mathbf{x})] 	\leq  (1-\fraca{a\lambda(t-n)}{4}) a_{n,t}^{\lambda,p-1}|Y^{\lambda}_n(\mathbf{x})|^{2p} + b_{n,t}^{\lambda,p} \nonumber\\
							&+\lambda(t-n)2^{2p-2}p(2p-1)(\fraca{d}{\beta}) b_{n,t}^{\lambda,p-1}
+\lambda^p(t-n)^p2^{3p-3}(p(2p-1))^{p+1}\left(\fraca{d}{\beta}\right)^p \nonumber \\
\leq & (1-a\lambda(t-n))|Y^{\lambda}_n(\mathbf{x})|^{2p} + \lambda(t-n)a\left(M(p,d)|x_n|^{2p} + \widehat{M}(p,d) \right),
\end{align}
where the constants $M(p,d)$ and $\widehat{M}(p,d)$ are defined in \eqref{eq:def:M_and_Mhat}. 
Moreover, on $\{|Y^{\lambda}_n(\mathbf{x})| < \sqrt{d}\tilde{M}(p)\}$ the inequality \eqref{long_calc_mom} yields again
\begin{multline} \label{long_calc_mom3}
\E[|Y_t^{\lambda}(\mathbf{x})|^{2p}|Y^{\lambda}_n(\mathbf{x})]
\leq  (1-a\lambda(t-n))|Y^{\lambda}_n(\mathbf{x})|^{2p} + \lambda(t-n)a\left(M(p,d)|x_n|^{2p} + \widehat{M}(p,d) \right)
\end{multline}
Eq.~\eqref{eq:moment_SGLD_2p} follows immediately from \eqref{long_calc_mom3} and \eqref{long_calc_mom2}. \end{proof}
\begin{remark}
{\rm One notes here that $\left(\E[|Y^{\lambda}_t(\mathbf{x})|^{2p}] \right)^{1/(2p)}$ is of order $\sqrt{d}$, where $d$ denotes the dimension of the problem.}
\end{remark}

\begin{corollary} \label{aux_proc_conts_fourth_and_Y}
Assume \Cref{imit}, \Cref{assum:lip} and \Cref{assum:dissipativity}.  For each $0<\lambda\leq \lambda_{\max}$ and  $0\leq s\leq t$, let $\tzeta{\lambda}{t}{s}{\mathbf{x}}{\theta}$ be the solution of \eqref{eq:aux_proc_conts}
with initial condition $\theta$.
Then for each $k\geq 1$,
\begin{multline}\label{fourth_mom_aux_proc_and_Y}
\E[V_4(\tzeta{\lambda}{kT}{(k-1)T}{\mathbf{x}}{Y^{\lambda}_{(k-1)T}(\mathbf{x})})] \le 2  \rme^{-a} (1-a\lambda)^{(k-1)T}
\E|\theta_0|^{4} \\+ 2  \rme^{-a} \left\{1+ a\lambda  M(2,d) \sum_{j=0}^{(k-1)T-1} \left(1-a\lambda\right)^j|x_{(k-1)T -1 - j}|^{4}
+   \widehat{M}(2,d)  \right\} + 3 \mathrm{v}_4(\overline{M}(4)),
\end{multline}
where the constants $M(2,d)$ and $\widehat{M}(2,d)$ are given by \eqref{eq:def:M_and_Mhat} and \eqref{eq:def:M_tilde} with $p=2$.
\end{corollary}
\begin{proof}
A direct consequence of Lemma~\ref{lem:aux_proc_conts_V_p}, 
\eqref{eq:moment_SGLD_2p} and the fact that $C_6(4) = a$.
\end{proof}

We now define a continuous-time filtration $(\mathcal{H}_t)_{t \geq 0}$ that encapsulates
the information flow of $X_{n}$, $n\in\mathbb{N}$ as well as all the ``auxiliary'' randomness
of the Brownian motion ${B}_{t}$, $t\in\mathbb{R}_{+}$. We also introduce the corresponding decreasing family
of $\sigma$-algebras $(\mathcal{H}_t^+)_{t \geq 0}$.
\begin{equation}
\label{eq:definition-sigma-field-H}
\mathcal{H}_t \eqdef \mathcal{F}_{\infty}\vee \mathcal{G}_{\lfloor t\rfloor} \quad \text{and}
\quad \mathcal{H}_t^{+}:=\mathcal{G}^+_{\lfloor t\rfloor}, \quad t\in\rset_+ \,
\end{equation}
where $(\mathcal{G}_n,\mathcal{G}^+_n)_{n \in \nset}$ are as in Assumption
\ref{assum:lmiu}.

We introduce another auxiliary process that play a prominent r\^{o}le in the sequel.
Let $\tZ{\lambda}{t}{s}{\vartheta}$, $t\geq s$
denote the solution of the SDE
\begin{equation}\label{averrage}
\rmd \tZ{\lambda}{t}{s}{\vartheta}=-\lambda h(\tZ{\lambda}{t}{s}{\vartheta})\, \rmd t+\{2\lambda \beta^{-1}\}^{1/2}
\rmd \tilde{B}^{\lambda}_t,
\end{equation}
with initial condition $\tZ{\lambda}{s}{s}{\vartheta}:=\vartheta$ for some $\mathcal{H}_s^{\lambda}$-measurable
random variable $\vartheta$. Note that $L^{\lambda}_t= \tZ{\lambda}{t}{0}{\theta_0}$.
At this point, we introduce
\begin{equation}
\label{eq:definition:T}
T:=\lfloor 1/{\lambda}\rfloor \,,
\end{equation}
which is used for the creation of a suitable set of grid points.
Fix $n \in \nset$ and define
for any $t\in\coint{nT, \infty}$
\begin{equation}
\label{eq:definition-overline-Z}
\bZ{\lambda}{nT}{t}= \tZ{\lambda}{t}{nT}{\tilde{Y}_{nT}^{\lambda}(\mathbf{X})} \,.
\end{equation}
Note that $\bZ{\lambda}{nT}{t}$ is $\mathcal{H}_{nT}$-measurable for all $t \ge nT$.

\begin{lemma}\label{lem:z_moment_bounds}
Assume \Cref{imit}, \Cref{assum:lip} and \Cref{assum:dissipativity}. For any integers $p\geq 2$, $n \in \nset$, $\lambda > 0$ and $t \geq nT$,
\begin{equation}\label{z_moment_p1}
\E[V_{p}(\bZ{\lambda}{nT}{t})] \leq \rme^{-\lambda C_6(p) t} \E[V_{p}(\theta_0)]+ 6\mathrm{v}_p(\overline{M}(p)),
\end{equation}
where  $\overline{M}(p)$ and $V_p$ are defined in \eqref{eq:definition-barM-p} and \eqref{eq:definition-Vp}.
\end{lemma}
\begin{proof}
By taking into consideration \eqref{eq:genlib} and by arguing as in Lemma~\ref{lem:aux_proc_conts_V_p}, one obtains $\E[V_{p}(\bZ{\lambda}{nT}{t})] \leq  \rme^{-\lambda C_6(p)(t-nT)} \E[V_{p}(\tilde{Y}_{nT}^{\lambda}(\mathbf{X}))]+ 3\mathrm{v}_p(\overline{M}(p))$. Hence, the desired result follows from Corollary~\ref{cor:SDE_SGLDsmoments}.
\end{proof}

Control of the supremum process of $\bZ{\lambda}{nT}{t}$ is an essential ingrediant in the proof of Lemma \ref{vizier} below.

\begin{corollary} \label{cor:moment_sup_process}
Assume \Cref{imit}, \Cref{assum:lip} and \Cref{assum:dissipativity}. For any integer $p\geq 2$,
\begin{equation}\label{eq:moment_sup_process}
 \E[\sup\nolimits_{nT \le t\le (n+1) T}V_{p}(\bZ{\lambda}{nT}{t})] \le 3 \rme^{-\lambda C_6(p)nT} \E[V_p(\theta_0)]+ C_{12}(p),
\end{equation}
where $T$ and $\overline{M}(2p)$  are given in \eqref{eq:definition:T} and \eqref{eq:definition-barM-p} respectively, and
\begin{equation}
\label{eq:definition-C12}
C_{12}(p) \eqdef 9 (1 +  (3ap)^{1/2}/2) \operatorname{v}_{p}(\overline{M}(2p)) \,.
\end{equation}
\end{corollary}
\begin{proof}
For any $n\in \nset$, $q\ge2$ and any bounded stopping time $\tau_{{\scriptscriptstyle n}} \ge nT$ (a.s.), arguing as in Lemma~\ref{lyapp} results in
\begin{align*}
\CPE{V_q(\bZ{\lambda}{nT}{\tau_{{\scriptscriptstyle n}}})}{\mathcal{H}_{nT}} &\le V_q(\tilde{Y}_{n T}^{\lambda}(\mathbf{X})) + \CPE{\int\nolimits_{nT}^{\tau_{{\scriptscriptstyle n}}}\left(-\lambda C_6(q) V_q(\bZ{\lambda}{nT}{s}) +\lambda C_7(q)\right)\, \rmd s}{\mathcal{H}_{nT}} \\
& \le V_q(\tilde{Y}_{n T}^{\lambda}(\mathbf{X})) + \lambda C_7(q)\CPE{(\tau_{{\scriptscriptstyle n}}-nT)}{\mathcal{H}_{nT}} \, .
\end{align*}
Then, according to Lenglart's domination inequality, see \cite[Chapter~IV, Proposition~4.7]{revuz:yor:1999}, with dominating process
\[
A_t := V_{q}\left(\tilde{Y}_{n T}^{\lambda}(\mathbf{X})\right)+\lambda C_{7}(q) (t-n T), \mbox{ for any } t\ge nT,
\]
one obtains, for any $k\in(0,1)$,
\[
\E\bigg[\left(\sup\nolimits_{nT \le t\le (n+1) T}V_{q}(\bZ{\lambda}{nT}{t})\right)^k\bigg]
\leq \frac{2-k}{1-k} \E[A_{(n+1)T}^k]
\]
Thus, using $(a + b)^k \leq a^k + b^k$ for any $a,b \geq 0$ and $k \in \ooint{0,1}$, we get
\[
\E\bigg[\left(\sup\nolimits_{nT \le t\le (n+1) T}V_{q}(\bZ{\lambda}{nT}{t})\right)^k\bigg] \le \frac{2-k}{1-k} \left\{ \E\bigg[ \left( V_q(\tilde{Y}_{n T}^{\lambda}(\mathbf{X})) \right)^k\bigg] +  C_7^k(q)  (\lambda T)^k \right\} \,.
\]
Consequently, for $k=1/2$ and $q=2p$ and in view of Corollary \ref{cor:SDE_SGLDsmoments}, the desired result holds.
\end{proof}

\subsection{Contraction estimates}

A crucial contraction property is formulated in the next theorem, based on the deep results of \cite{eberle:guillin:zimmer:Trans:2019}.

\begin{proposition}
\label{prop:contra} Let $(L_t')_{t\in\rset_+}$  be the solution of \eqref{kako} with initial condition $L_0'=\theta_0'$ which is independent of $\mathcal{F}_{\infty}$
and satisfies $\theta_0' \in L^2$. Then,
\begin{equation}\label{karako}
w_{1,2}(\mathcal{L}(L_t),\mathcal{L}(L_t'))\leq C_9 \rme^{-C_8 t}
w_{1,2}(\mathcal{L}(\theta_0),\mathcal{L}(\theta_0')),\ t\in\rset_+,
\end{equation}
where the constants $C_8$ and $C_9$ are given explicitly in Lemma \ref{contractionconst} and $w_{1,2}$ is defined in \eqref{eq:definition-w-1-p}. Fix a positive integer $m$. Suppose, for any $t>m$, $\tzeta{\lambda}{t}{m}{\mathbf{x}}{\tilde{\theta}}$  and $\tzeta{\lambda}{t}{m}{\mathbf{x}}{\tilde{\theta}'}$ are the solutions of \eqref{eq:aux_proc_conts} with initial conditions $\tilde{\theta}$, $\tilde{\theta}' \in L^2$, which are independent of $\mathcal{F}_{\infty}$. Then, for any $t>m$, we get
\begin{equation}\label{eq:contra_w_1_2}
w_{1,2}(\mathcal{L}(\tzeta{\lambda}{t}{m}{\mathbf{x}}{\tilde{\theta}}), \mathcal{L}(\tzeta{\lambda}{t}{m}{\mathbf{x}}{\tilde{\theta}'})) \le C_9  \rme^{-C_8\lambda(t-m)}   w_{1,2}(\mathcal{L}(\tilde{\theta}), \mathcal{L}(\tilde{\theta}')) \, .
\end{equation}
\end{proposition}
\begin{proof}
We first treat $L_t$, $L_t'$.
\cite[Assumption~2.1]{eberle:guillin:zimmer:Trans:2019} holds with $\kappa$ constant (and equal to $K_1$) due to \Cref{assum:lip}. \cite[Assumption~2.5]{eberle:guillin:zimmer:Trans:2019} holds due to \eqref{solymos}.  \cite[Assumption~2.2]{eberle:guillin:zimmer:Trans:2019} holds with $V=V_2$ due to Lemma~\ref{lyapp} (note that in that paper the diffusion coefficient is assumed to be $1$ while in our case it is $\sqrt{2/\beta}$
but this does not affect the validity of the arguments, only the values of the constants).
Thus, in view of  \cite[Corollary~2.3]{eberle:guillin:zimmer:Trans:2019},
\[
\mathcal{W}_{\rho_2}
(\mathcal{L}(L_t),\mathcal{L}(L_t'))\leq \rme^{-C_8 t}
\mathcal{W}_{\rho_2}(\mathcal{L}(\theta_0),\mathcal{L}(\theta_0')),\ t\in\rset_+,
\]
where $C_8$ is given  in Lemma \ref{contractionconst} and the functional
$\mathcal{W}_{\rho_2}$ comes from \cite{eberle:guillin:zimmer:Trans:2019} with the choice $V:=V_2$,
for $\mu,\nu\in\mathcal{P}_{\, V_2}(\rset^d)$
\begin{equation}\label{lajjja}
\mathcal{W}_{\rho_2}(\mu,\nu) : =\inf_{\zeta\in\mathcal{C}(\mu,\nu)}\int_{\rset^d}\int_{\rset^d} f( |\theta-\theta'|)(1+ \epsilon V_2(\theta) + \epsilon V_2(\theta'))\zeta(\rmd \theta \rmd \theta'),
\end{equation}
where $f$ is a concave, bounded and non-decreasing continuous function and $\epsilon$ is a positive constant, for more details see   \cite[Section~5]{eberle:guillin:zimmer:Trans:2019}. Consequently, by using the definition of $\mathcal{W}_{\rho_2}$, one obtains
\begin{equation}\label{lajja}
C_{10} w_{1,2}(\mu,\nu)\leq
\mathcal{W}_{\rho_2}
(\mu,\nu)\leq C_{11} w_{1,2}(\mu,\nu),\quad  \mu,\nu\in\mathcal{P}_{\, V_2}(\rset^d),
\end{equation}
where $C_{10},C_{11}$ are calculated in Lemma \ref{contractionconst} below. Statement \eqref{karako} follows with $C_9=C_{11}/C_{10}$.

The same approach is used for  $\tzeta{\lambda}{t}{m}{\mathbf{x}}{\tilde{\theta}}$  and $\tzeta{\lambda}{t}{m}{\mathbf{x}}{\tilde{\theta}'}$, with the only difference being that we derive first the contraction on an interval of length at most one, since the contribution from the data sequence, through $x_{\lfloor t\rfloor}$, remains constant and thus, the drift coefficient remains autonomous for such an interval. More concretely,  \cite[Assumption~2.1]{eberle:guillin:zimmer:Trans:2019} holds in this case too with $\kappa$ constant and equal to $K_1$ due to \Cref{assum:lip}. \cite[Assumption~2.2]{eberle:guillin:zimmer:Trans:2019} is true with $V=V_2$ due to Lemma
\ref{lyapp}. Note that the statements in these Assumptions are uniform in $x$ (and thus identical for different values of $x_{\lfloor t\rfloor}$). Finally,  \cite[Assumption~2.5]{eberle:guillin:zimmer:Trans:2019} is also true due to \eqref{solymos}. Thus, the results of  \cite[Corollary 2.3]{eberle:guillin:zimmer:Trans:2019} apply in this case, too, and one concludes that
\begin{align} \label{eq:1_step_contra}
&\mathcal{W}_{\rho_2}
(\mathcal{L}(\tzeta{\lambda}{t}{m}{\mathbf{x}}{\tilde{\theta}}), \mathcal{L}(\tzeta{\lambda}{t}{m}{\mathbf{x}}{\tilde{\theta}'})) \nonumber \\
&\quad =
\mathcal{W}_{\rho_2}(\mathcal{L}(\tzeta{\lambda}{t}{\lfloor t\rfloor}{\mathbf{x}}{\tzeta{\lambda}{\lfloor t\rfloor}{m}{\mathbf{x}}{\tilde{\theta}}}), \mathcal{L}(\tzeta{\lambda}{t}{\lfloor t\rfloor}{\mathbf{x}}{\tzeta{\lambda}{\lfloor t\rfloor}{m}{\mathbf{x}}{\tilde{\theta}'}})) \nonumber \\
& \quad \leq  \rme^{-C_8 \lambda(t - \lfloor t\rfloor)}
\mathcal{W}_{\rho_2}\left(\mathcal{L}(\tzeta{\lambda}{\lfloor t\rfloor}{m}{\mathbf{x}}{\tilde{\theta}}),\mathcal{L}(\tzeta{\lambda}{\lfloor t\rfloor}{m}{\mathbf{x}}{\tilde{\theta}'})\right) \nonumber \\
& \quad \leq \rme^{-C_8 \lambda(t-(\lfloor t\rfloor-1))}
\mathcal{W}_{\rho_2}(\mathcal{L}(\tzeta{\lambda}{\lfloor t\rfloor - 1}{m}{\mathbf{x}}{\tilde{\theta}}),\mathcal{L}(\tzeta{\lambda}{\lfloor t\rfloor - 1}{m}{\mathbf{x}}{\tilde{\theta}'})) \nonumber \\
&\quad\le \ldots \nonumber \\
&\quad \le  \rme^{-C_8\lambda(t-m)}   \mathcal{W}_{\rho_2}(\mathcal{L}(\tilde{\theta}), \mathcal{L}(\tilde{\theta}')).
\end{align}
Observing as above that $\mathcal{W}_{\rho_2}$ is controlled from above and below by multiples of $w_{1,2}$, \eqref{eq:1_step_contra} yields the result.
\end{proof}

\subsection{The core lemmas}

Our arguments for handling the second term in \eqref{alambda} rest upon Lemmas \ref{vizier} and \ref{intermediate} below.
As a preparation, we first recall two lemmas: one on regular
versions and one on moment estimates that are closely related to the conditional $L$-mixing property.

\begin{lemma}\label{haa} For each
$n\in\nset$, there exists a measurable function $ h :\Omega\times \coint{nT,\infty}\times
\rset^d\to\rset^d$
such that, for each $t\geq nT$ and $\theta\in\rset^d$, $h_{nT,t}(\theta)(\omega)$
is a version of $\CPE{H(\theta,X_{\lfloor t\rfloor})}{\mathcal{H}_{nT}}$
for almost every $\omega\in\Omega$, $\theta\to h_{nT,t}(\theta)(\omega)$ is continuous.
\end{lemma}
\begin{proof}
As $h_{nT,t}$, $t\in \coint{k,k+1}$ can be assumed constant
for each $k\in\nset$, it suffices to prove the existence of a measurable
$h_{nT,k}:\Omega\times\rset^d\to\rset^d$ which is continuous in its second variable, for each fixed $k$. This follows from \cite[Lemma~8.5]{barkhagen2018stochastic}.
\end{proof}

\begin{lemma}\label{lem:h_minus_cond_h}
Assume \Cref{assum:lip} and \Cref{assum:lmiu} and let $p \geq 1$. Then,
	\[
\sup_{n\in\nset}\E^{1/p}\left[\left(
\sum\nolimits_{k=nT}^{\infty} \sup\nolimits_{\theta\in\mathbb{R}^d}\left\Vert h_{k,nT}(\theta)- h(\theta)\right\Vert{}
\right)^{p}\right]\leq 2K_2 \Gamma_p^0(X) ,
    \]
where $\Gamma_p^0(X)$ is defined in \eqref{eq:definition-Gamma}.
\end{lemma}
\begin{proof}
See \cite[Lemma~4.9]{barkhagen2018stochastic}.
\end{proof}

Now we present the first core lemma.

\begin{lemma}\label{vizier} Assume \Cref{imit}, \Cref{assum:lip} and \Cref{assum:dissipativity} hold. There is
$C_{13}$ such that,
for each $0<\lambda\leq \lambda_{\max}$, and for all $t\in [nT,(n+1)T]$,
\[
{W}_{2}(\mathcal{L}(\tilde{Y}^{\lambda}_t(\mathbf{X})),\mathcal{L}(\bZ{\lambda}{nT}{t})) \leq C_{13}\lambda^{1/2}
[\rme^{-an/4}\E^{1/2}[V_{2}(\theta_{0})]+1].
\]
\end{lemma}
\begin{proof}
Fix $t\in [nT,(n+1)T]$. Let us estimate, using \Cref{assum:lip},
\begin{align*}
&\left|\tilde{Y}^{\lambda}_t(\mathbf{X})-\bZ{\lambda}{nT}{t}\right| \leq  \lambda \left\vert\int_{nT}^t
\left[H(\tilde{Y}^{\lambda}_s(\mathbf{X}),X_{\lfloor s\rfloor})-h(\bZ{\lambda}{nT}{s})\right]\, \rmd s\right\vert \\
&\leq
\lambda\int_{nT}^t
\left|H(\tilde{Y}^{\lambda}_s(\mathbf{X}),X_{\lfloor s\rfloor})-H(\bZ{\lambda}{nT}{s},
X_{\lfloor s\rfloor})\right|\, \rmd s
+ \lambda\left\vert\int_{nT}^t
\left[H(\bZ{\lambda}{nT}{s},X_{\lfloor s\rfloor})-h_{nT,s}(\bZ{\lambda}{nT}{s})\right]\, \rmd s\right\vert \\
&+ \lambda\int_{nT}^t
\left|h_{nT,s}(\bZ{\lambda}{nT}{s})-h(\bZ{\lambda}{nT}{s})\right|\, \rmd s \\
&\leq \lambda K_1 \int_{nT}^t\left|\tilde{Y}^{\lambda}_s(\mathbf{X})-\bZ{\lambda}{nT}{s}\right|\, \rmd s + \lambda A + \lambda B
\end{align*}
where $h_{nT,s}$ is defined in Lemma \ref{haa} and
\begin{align*}
A &\eqdef \sup_{u\in [nT,(n+1)T]}\left|\int_{nT}^u
\left[H(\bZ{\lambda}{nT}{s},X_{\lfloor s\rfloor})-h_{nT,s}(\bZ{\lambda}{nT}{s})\right]\, \rmd s\right|\\
B &\eqdef \int_{nT}^{\infty}
\sup_{\theta \in \rset^d} \left|h_{nT,s}(\theta)-h(\theta)\right|\ \rmd s,
\end{align*}
Now let us apply Gr\"onwall's lemma and take the square of both sides.
Using the elementary $(x+y)^2\leq 2(x^{2}+y^{2})$, $x,y\geq 0$,
we arrive at
\begin{equation} \label{eq:trafo}
\left|\tilde{Y}^{\lambda}_t(\mathbf{X})-\bZ{\lambda}{nT}{t}\right|^{2}
\leq  2\lambda^{2} \rme^{2 K_1\lambda T} \{A^2 + B^2 \}
\end{equation}
Introduce for all $i\in\nset$ the events
\begin{equation*}
F^{nT}_i:=\{i\leq \sup\nolimits_{s\in [nT,(n+1)T]}|\bZ{\lambda}{nT}{s}|<i+1\} \, ,
\end{equation*}
which are $\mathcal{H}_{nT}$-measurable.

We apply below Theorem~\ref{estim} in the following setting. Let $\mathcal{R}_{s}=\mathcal{H}_{nT+s}$
and $\mathcal{R}^{+}_{s}=\mathcal{H}_{nT+s}^{+}$ for $s\in\mathbb{R}_{+}$. Furthermore, let $W_{s}=W^{nT,i}_{s-nT}$
where we define
\[
W^{nT,i}_s:=\left(H(\bZ{\lambda}{nT}{s},X_{\lfloor  s \rfloor})-h_{nT,s}(\bZ{\lambda}{nT}{s})\right) \indi{F^{nT}_i},\quad
\quad i\in \nset, s \geq nT.
\]
Clearly, for $s \geq nT$, $\CPE{W^{nT,i}_s}{\mathcal{H}_{nT}}=0$. We now estimate the quantities $M_{p}(\mathbf{W})$, $\Gamma_{p}(\mathbf{W})$
appearing in Theorem~\ref{estim}.

For each fixed $\theta$, Lemma~\ref{lem:below} implies that the auxiliary process
$\tilde{W}^{\theta}_{s}:=H(\theta,X_{\lfloor nT+ s \rfloor})\indi{F^{nT}_i}$, $s\in\mathbb{R}_{+}$
satisfies
\[
M_{p}(\mathbf{\tilde{W}}^{\theta})\leq K_1 i + K_2 M^{nT}_{p}(\mathbf{X})+H^*\
\]
as well as
\[
\Gamma_{p}(\mathbf{\tilde{W}}^{\theta})\leq 2K_2 \Gamma^{nT}_{p}(\mathbf{X}).
\]
Hence Lemma~\ref{lemma_63} guarantees that we can plug
in the $\mathcal{H}_{nT}$-measurable process $\bZ{\lambda}{nT}{s}$ into $\tilde{W}^{\theta}_{s}$, getting
\[
M_{p}(\mathbf{\hat{W}})\leq K_1 i + K_2 M^{nT}_{p}(\mathbf{X})+H^*,\
\Gamma_{p}(\mathbf{\hat{W}})\leq 2K_2 \Gamma^{nT}_{p}(\mathbf{X})
\]
for the process defined by
$$
\hat{W}_{s}:=H(\bZ{\lambda}{nT}{s},X_{\lfloor nT+ s \rfloor})\indi{F^{nT}_i},\ s\in\mathbb{R}_{+}.
$$
Finally, by \cite[Remark~A.4]{chau:kumar:rasonyi:sabanis:2019} (or after a moment's reflection), we find that
\[
M_{p}(\mathbf{W})\leq 2[K_1 i + K_2 M^{nT}_{p}(\mathbf{X})+H^*],\
\Gamma_{p}(\mathbf{W})\leq 2K_2 \Gamma^{nT}_{p}(\mathbf{X}).
\]
Applying Theorem \ref{estim} with $r:=3$, we obtain
\begin{align*}
&\CPE[1/2]{\sup_{u\in [nT,(n+1)T]}
 \left|\int_{nT}^u
[H(\bZ{\lambda}{nT}{s},X_{\lfloor s\rfloor})-h_{nT,s}(\bZ{\lambda}{nT}{s})]
\, \rmd s\right|^{2} \indi{F^{nT}_i}}{\mathcal{H}_{nT}}\\
&\quad \leq  \CPE[1/3]{\sup_{u\in [nT,(n+1)T]}
 \left|\int_{nT}^u
[H(\bZ{\lambda}{nT}{s},X_{\lfloor s\rfloor})-h_{nT,s}(\bZ{\lambda}{nT}{s})]
\, \rmd s\right|^{3} \indi{F^{nT}_i}}{\mathcal{H}_{nT}}\\
& \quad \leq
2C'(3)\sqrt{T}[K_1 i + K_2 M^{nT}_{3}(\mathbf{X})+K_2\Gamma^{nT}_{3}(\mathbf{X})+H^*]\indi{F^{nT}_i}\\
& \quad \leq
20\sqrt{T}[K_1 (1+\sup\nolimits_{s\in [nT,(n+1)T]}|\bZ{\lambda}{nT}{s}|) + K_2 M^{nT}_{3}(\mathbf{X})+
K_2\Gamma^{nT}_{3}(\mathbf{X})+H^*]\indi{F^{nT}_i},
\end{align*}
noting that the constant $C'(3)$ appearing in Theorem \ref{estim} satisfies $C'(3)\leq 10$.
We can then estimate, noting $C_{6}(2)=a/2$ (see Lemma \ref{lyapp}),
\begin{align*}
 \E^{1/2}\left[A^{2}\right]
\leq& 20\sqrt{T}[K_{1}(E^{1/2}[\sup\nolimits_{s\in [nT,(n+1)T]}|\bZ{\lambda}{nT}{s}|^{2}]+1)
+K_{2}E^{1/2}[(M^{nT}_{3})^{2}] \\
& +K_{2}E^{1/2}[(\Gamma^{nT}_{3})^{2}]+H^{*}]\\
\leq& 20
\lambda^{-1/2}[K_{1}(\sqrt{3}(\rme^{-\lambda anT/4} \E^{1/2}[V_{2}(\theta_0)] +1 + {C}^{1/2}_{12}(2)) +K_{2}E^{1/2}[(M^{nT}_{3})^{2}]\\
&
+K_{2}E^{1/2}[(\Gamma^{nT}_{3})^{2}]+H^{*}]
\end{align*}
using Corollary~\ref{cor:moment_sup_process} with the choice $p=2$.

Finally, for any $t \in [nT,(n+1)T]$ and $\lambda \in \ocint{0,\lambda_{\max}}$, using \eqref{eq:trafo}
and Lemma~\ref{lem:h_minus_cond_h} we get
\begin{align}
{W}_{2}(\mathcal{L}(\tilde{Y}^{\lambda}_t(\mathbf{X})),\mathcal{L}(\bZ{\lambda}{nT}{t})) \leq& \E^{1/2}\left|\tilde{Y}^{\lambda}_t(\mathbf{X})-\bZ{\lambda}{nT}{t}\right|^{2} \nonumber\\
\leq&  20 \sqrt{2} \rme^{K_1} \lambda^{1/2}
[K_{1}(\sqrt{3}(\rme^{-\lambda anT/4} \E^{1/2}[V_{2}(\theta_0)]+1 +H^{*} + {C}^{1/2}_{12}(2))\nonumber\\
&
+K_{2}E^{1/2}[(M^{nT}_{3})^{2}]+K_{2}E^{1/2}[(\Gamma^{nT}_{3})^{2}]+2\lambda_{\mathrm{max}}K_{2}\Gamma_{2}^{0}(\mathbf{X})]
\label{pad}.
\end{align}
So we can conclude choosing
\begin{eqnarray*}
C_{13} &=& 20 \sqrt{2} \rme^{K_1}[K_{1}\sqrt{3} +K_{1}(1+{C}^{1/2}_{12}(2))+\\
&+& K_{2}
E^{1/2}[(M^{nT}_{3})^{2}]+K_{2}E^{1/2}[(\Gamma^{nT}_{3})^{2}]+H^{*}+2\lambda_{\mathrm{max}}K_{2}\Gamma_{2}^{0}(\mathbf{X})].
\end{eqnarray*}

\end{proof}
The second core lemma follows from the first and from Proposition \ref{prop:contra}.
\begin{lemma} \label{intermediate}
\label{lem:bound-overline-Z-L}
Assume \Cref{imit}, \Cref{assum:lip} and \Cref{assum:dissipativity}. For each $0<\lambda\leq\lambda_{\mathrm{max}}$,
$n\in\nset$ and $t\in \coint{nT,(n+1)T}$,
\begin{equation*}
W_1(\mathcal{L}(\bZ{\lambda}{nT}{t}),\mathcal{L}(L^{\lambda}_t)) \leq C_{14}[1+\rme^{-\min\{C_{8},a/4\}n/2}\E^{3/4}[V_{4}(\theta_0)]] \sqrt{\lambda}
\end{equation*}
for a suitable $C_{14}$, explicitly given in the proof.
\end{lemma}
\begin{proof}
Using telescopic sums,
\eqref{eq:lucia} and Proposition~\ref{prop:contra}, we get
\begin{align}\label{puszta}
 W_1(\mathcal{L}(\bZ{\lambda}{nT}{t}),\mathcal{L}(L^{\lambda}_t))
\leq &  \sum _{k=1}^n W_1\left(\mathcal{L}(\tZ{\lambda}{t}{kT}{\tilde{Y}^{\lambda}_{kT}(\mathbf{X})}),
\mathcal{L}(\tZ{\lambda}{t}{(k-1)T}{\tilde{Y}^{\lambda}_{(k-1)T}(\mathbf{X})})\right)\\ \nonumber
\leq& \sum_{k=1}^n w_{1,2}(\mathcal{L}(\tZ{\lambda}{t}{kT}{\tilde{Y}^{\lambda}_{kT}(\mathbf{X})}),
\mathcal{L}(\tZ{\lambda}{t}{kT}{\tZ{\lambda}{kT}{(k-1)T}{\tilde{Y}^{\lambda}_{(k-1)T}(\mathbf{X})}}))\\
\nonumber
\leq&
C_9\sum_{k=1}^n \exp\left(-C_8(n-k)\right)
w_{1,2}(\mathcal{L}(\tilde{Y}^{\lambda}_{kT}(\mathbf{X})),\mathcal{L}(\bZ{\lambda}{(k-1)T}{kT})).
\end{align}

Using the definitions \eqref{eq:definition-w-1-p} and \eqref{eq:definition-W-1} of $w_{1,2}$ and ${W}_2$, we get
from the Cauchy inequality that
\begin{align*}
w_{1,2}(\mathcal{L}(\tilde{Y}^{\lambda}_{kT}(\mathbf{X})),\mathcal{L}(\bZ{\lambda}{(k-1)T}{kT}))
\leq&
{W}_{2}(\mathcal{L}(\tilde{Y}^{\lambda}_{kT}(\mathbf{X})),\mathcal{L}(\bZ{\lambda}{(k-1)T}{kT})) \\
&\times  [1+ \{\E[V_4(\tilde{Y}^{\lambda}_{kT}(\mathbf{X}))]\}^{1/2} + \{\E[V_4(\bZ{\lambda}{(k-1)T}{kT})]\}^{1/2} ].
\end{align*}
Corollary \ref{cor:SDE_SGLDsmoments}, Lemma \ref{lem:z_moment_bounds} and Lemma \ref{vizier} imply that
\begin{align*}
w_{1,2}(\mathcal{L}(\tilde{Y}^{\lambda}_{kT}(\mathbf{X})),\mathcal{L}(\bZ{\lambda}{(k-1)T}{kT}))
\leq &
C_{13}\lambda^{1/2}
[\rme^{-a(k-1)/4}\E^{1/2}[V_{2}(\theta_{0})]+1]\\
\times &[1+2\rme^{-ak/2}\{\E[V_4(\theta_{0})]\}^{1/2}+\sqrt{3}\mathrm{v}_2(\overline{M}(4)) + \sqrt{6}\mathrm{v}_2(\overline{M}(4))]
\end{align*}
since $C_{6}(4)=a$.
For each $y\geq 0$ and $\alpha>0$, $e^{-\alpha y}(y+1)\leq 1+1/\alpha$.
In the estimations below we apply this latter observation with $\alpha=\min(C_8,a/4)/2$ and $y=n-1$.
Noticing that $\E^{1/2}[V_{2}(\theta_{0})]\leq \E^{1/4}[V_{4}(\theta_{0})]$, we can proceed as
\begin{eqnarray*} & &
\sum_{k=1}^n \exp\left(-C_8(n-k)\right)
w_{1,2}(\mathcal{L}(\tilde{Y}^{\lambda}_{kT}(\mathbf{X})),\mathcal{L}(\bZ{\lambda}{(k-1)T}{kT}))\\
&\leq&  C_{13}\lambda^{1/2} [2\E^{3/4}[V_4(\theta_{0})]+\E^{1/2}[V_4(\theta_{0})]+
(5\mathrm{v}_2(\overline{M}(4))+1)\E^{1/4}[V_{4}(\theta_{0})]]\\
&\times&
\sum_{k=1}^n \exp\left(-\min\{C_8,a/4\}(n-k+k-1)\right)\\
&+& C_{13}\lambda^{1/2}\frac{5\mathrm{v}_2(\overline{M}(4))+1}{1-\rme^{-C_{8}} }\\
&\leq& C_{13}\lambda^{1/2}n\exp\left(-\min\{C_8,a/4\}(n-1)\right)
[(5\mathrm{v}_2(\overline{M}(4))+4)\E^{3/4}[V_4(\theta_{0})] +5\mathrm{v}_2(\overline{M}(4))+1+1]\\
&+& C_{13}\lambda^{1/2}\frac{5\mathrm{v}_2(\overline{M}(4))+1}{1-\rme^{-C_{8}} }\\
&\leq& C_{13}\lambda^{1/2}\exp\left(-\min\{C_8,a/4\}(n-1)/2\right)\left(1+\frac{2}{\min\{C_8,a/4\}}\right)\\
&\times& [(5\mathrm{v}_2(\overline{M}(4))+4)\E^{3/4}[V_4(\theta_{0})] +5\mathrm{v}_2(\overline{M}(4))+2]\\
&+& C_{13}\lambda^{1/2}\frac{5\mathrm{v}_2(\overline{M}(4))+1}{1-\rme^{-C_{8}} },
\end{eqnarray*}
and we can set
\begin{eqnarray*}
C_{14}&=&C_{9}C_{13}\left(1+\frac{2}{\min\{C_8,a/4\}}\right)\rme^{\min\{C_8,a/4\}}\left[5\mathrm{v}_2(\overline{M}(4))+4\right]\\
&+& C_{9}C_{13}\left[\frac{5\mathrm{v}_2(\overline{M}(4))+1}{1-\rme^{-C_{8}} }+
(5\mathrm{v}_2(\overline{M}(4))+2)\left(1+\frac{2}{\min\{C_8,a/4\}}\right)\right].
\end{eqnarray*}
\end{proof}

\begin{corollary} \label{cor_vizier}
For each $nT\leq t<(n+1)T$,
\begin{eqnarray*}
W_1(\mathcal{L}(L^{\lambda}_t),\mathcal{L}(\tilde{Y}^{\lambda}_t(\mathbf{X})))\leq C_{15}[1+
\exp({-\min\{C_8,a/4\}n/2})\E^{3/4}[V_{4}(\theta_{0})]]\sqrt{\lambda},
\end{eqnarray*}
for some $C_{15}$, explicitly given in the proof.
\end{corollary}
\begin{proof} Notice that $\E^{1/2}[V_{2}(\theta_{0})]\leq \E^{1/4}[V_{4}(\theta_{0})]$.
Putting together our previous estimations, we arrive at
\begin{eqnarray*}
& & W_1(\mathcal{L}(\tilde{Y}^{\lambda}_t(\mathbf{X})),\mathcal{L}(L^{\lambda}_t)) \\
&\leq&
W_1(\mathcal{L}(\tilde{Y}^{\lambda}_t(\mathbf{X})),\mathcal{L}(\bZ{\lambda}{nT}{t}))+
W_1(\mathcal{L}(\bZ{\lambda}{nT}{t}),\mathcal{L}(L^{\lambda}_t))\\
&\leq& \sqrt{\lambda}[\rme^{-\min\{C_8,a/4\}n/2}
[C_{14}\E^{3/4}[V_{4}(\theta_{0})]+C_{13}\E^{1/4}[V_{4}(\theta_{0})]]
+C_{14}+C_{13}]\\
&\leq& \sqrt{\lambda}[\rme^{-\min\{C_8,a/4\}n/2}\E^{3/4}[V_{4}(\theta_{0})](C_{14}+C_{13})+C_{14}+2C_{13}]
\end{eqnarray*}
so we can set $C_{15}:=C_{14}+2C_{13}$.
\end{proof}

\subsection{Entropy estimates}

We develop in this subsection the estimates that are necessary for coping with the third term in \eqref{alambda}.
Although the principal ideas are well-known, see e.g.\ \cite{dalalyan:tsybakov:2012,dalalyan:2017,durmus:moulines:2017},
the details require rather tedious technicalities since the estimates depend on the ``frozen'' data stream $\mathbf{x}=(x_{n})_{n\in\mathbb{N}}$.

\begin{lemma}\label{kl} Assume \Cref{imit}, \Cref{assum:lip} and \Cref{assum:dissipativity} hold. For each
$0<\lambda\leq \lambda_{\max}$ (see \eqref{eq:definition-lambda-max}) and $n\in\nset$ we have, for all $t\in (nT,(n+1)T]$ and $\mathbf{x}\in(\rset^m)^{\nset}$, that
\[
W_1(\mathcal{L}(\tilde{Y}^{\lambda}_t(\mathbf{x})),
\mathcal{L}(Y^{\lambda}_t(\mathbf{x})))\leq
\lambda^{1/2} \rme^{-\min(C_8,a) n/2}C_{17} \E[ |\theta_0|^4]+ \lambda^{1/2} C_{18}(\mathbf{x},n,\lambda)\ .
\]
where $C_{17}$ and $C_{18}(\mathbf{x},n,\lambda)$ are given by \eqref{eq:definition-C23_0} and \eqref{eq:definition-C23_1a} below, respectively.
\end{lemma}
\begin{proof}
Recall \eqref{eq:aux_proc_conts} and observe that $\tilde{Y}^{\lambda}_t(\mathbf{x})=\tzeta{\lambda}{t}{0}{\mathbf{x}}{\theta_0}$.
Using telescopic sums, we get for $t \in (nT,(n+1)T]$,
\begin{align*}
&W_1(\mathcal{L}(\tilde{Y}^{\lambda}_t(\mathbf{x})),
\mathcal{L}(Y^{\lambda}_t(\mathbf{x})))= W_1(\mathcal{L}(\tzeta{\lambda}{t}{0}{\mathbf{x}}{\theta_0}),
\mathcal{L}(Y^{\lambda}_t(\mathbf{x})))  \\
& \quad \le    \sum_{k=1}^n    W_1(\mathcal{L}(\tzeta{\lambda}{t}{kT}{\mathbf{x}}{Y^{\lambda}_{kT}(\mathbf{x})}),
\mathcal{L}(\tzeta{\lambda}{t}{(k-1)T}{\mathbf{x}}{Y^{\lambda}_{(k-1)T}(\mathbf{x})}))    \\
& \quad +  W_1(\mathcal{L}(\tzeta{\lambda}{t}{nT}{\mathbf{x}}{Y^{\lambda}_{nT}(\mathbf{x})}),
\mathcal{L}(Y^{\lambda}_t(\mathbf{x}))) \\
&\quad \le \sum_{k=1}^n    w_{1,2} (\mathcal{L}(\tzeta{\lambda}{t}{kT}{\mathbf{x}}{Y^{\lambda}_{kT}(\mathbf{x})}),
\mathcal{L}(\tzeta{\lambda}{t}{kT}{\mathbf{x}}{\tzeta{\lambda}{kT}{(k-1)T}{\mathbf{x}}{Y^{\lambda}_{(k-1)T}(\mathbf{x})}}))    \\
& \quad +  w_{1,2}(\mathcal{L}(\tzeta{\lambda}{t}{nT}{\mathbf{x}}{Y^{\lambda}_{nT}(\mathbf{x})}),
\mathcal{L}(Y^{\lambda}_t(\mathbf{x}))),
\end{align*}
where the domination of $W_1$ by $w_{1,2}$ is used, see \eqref{eq:lucia}.
In view of Proposition \ref{prop:contra}, and in particular inequality \eqref{eq:contra_w_1_2}, one obtains
\begin{multline}\label{W_1-estimation}
 W_1(\mathcal{L}(\tilde{Y}^{\lambda}_t(\mathbf{x})),
\mathcal{L}(Y^{\lambda}_t(\mathbf{x})))\le
   C_9 \sum_{k=1}^n \rme^{-C_8(n-k)}   w_{1,2}(\mathcal{L}(Y^{\lambda}_{kT}(\mathbf{x})),
\mathcal{L}(\tzeta{\lambda}{kT}{(k-1)T}{\mathbf{x}}{Y^{\lambda}_{(k-1)T}(\mathbf{x})})   \\ +  w_{1,2}(\mathcal{L}(\tzeta{\lambda}{t}{nT}{\mathbf{x}}{Y^{\lambda}_{nT}(\mathbf{x})}),
\mathcal{L}(Y^{\lambda}_t(\mathbf{x}))).
\end{multline}
At this point, one notes that due to Lemma \ref{lem:domination-w-1}, for any two probability measures $\mu$ and $\nu$ on $\mathcal{B}(\rset^d)$,
\begin{equation}\label{dune1}
w_{1,2}(\mu,\nu)\leq \sqrt{2} \left \{ 1 +  [\mu(V_4)]^{1/2} + [\nu(V_4)]^{1/2} \right\} \left\{ \operatorname{KL}(\mu,\nu) \right\}^{1/2} \, .
\end{equation}
where $\operatorname{KL}(\mu,\nu)$ denotes the Kullback-Leibler divergence. Thus
\begin{align} \label{w_{1,2}_and_KL}
w_{1,2} (\mathcal{L}(Y^{\lambda}_{kT}(\mathbf{x})),
\mathcal{L}(\tzeta{\lambda}{kT}{(k-1)T}{\mathbf{x}}{Y^{\lambda}_{(k-1)T}(\mathbf{x})}) \leq \sqrt{2\lambda} A^{1/2}_k B^{1/2}_k \leq \sqrt{\lambda/2} \{A_k + B_k\}
\end{align}
where
\begin{align}
\label{eq:definition-A-k}
A_k &\eqdef \lambda^{-1} \operatorname{KL}\left(\mathcal{L}(Y^{\lambda}_{kT}(\mathbf{x})), \mathcal{L}(\tzeta{\lambda}{kT}{(k-1)T}{\mathbf{x}}{Y^{\lambda}_{(k-1)T}(\mathbf{x})})\right) \\
\label{eq:definition-B-k}
B_k &\eqdef  \{1 + \E^{1/2}[V_4(Y^{\lambda}_{kT}(\mathbf{x}))] + \E^{1/2}[V_4(\tzeta{\lambda}{kT}{(k-1)T}{\mathbf{x}}{Y^{\lambda}_{(k-1)T}(\mathbf{x})})] \}^2
\end{align}
and $1\le k \le n$. For $a<b$, $\mathbf{C}[a,b]$ denotes the Banach space of $\rset^d$-valued
continuous functions on the interval $[a,b]$.
Let $\hat{\mathcal{Q}}_k$ denote the law of the process $\tzeta{\lambda}{s}{(k-1)T}{\mathbf{x}}{Y^{\lambda}_{(k-1)T}(\mathbf{x})}$,
$s\in [(k-1)T,kT]$ on $\mathbf{C}[(k-1)T,kT]$.
Similarly, let $\mathcal{Q}_k$ denote the law of $Y_s^{\lambda}(\mathbf{x})$, $s\in [(k-1)T,kT]$. Lemma \ref{lsKL} implies that these two probability laws are equivalent. Thus, in view of \eqref{eq:definition-KL}, one then calculates
\begin{align} \label{KL_estimate}
A_k &\leq \frac{1}{\lambda} \operatorname{KL}(\hat{\mathcal{Q}}_k\Vert\mathcal{Q}_k)\nonumber \\
&= \frac{1}{\lambda}\frac{1}{2}\frac{\beta}{2\lambda}\lambda^{2}\int_{(k-1)T}^{kT}
\E|H(Y^{\lambda}_{\lfloor s\rfloor}(\mathbf{x}),x_{\lfloor s\rfloor}) - H(Y^{\lambda}_s(\mathbf{x}),x_{\lfloor s\rfloor})|^2\, \rmd s\nonumber \\
&\leq \frac{\beta K_1^2}{4}\int_{(k-1)T}^{kT}
\E|Y^{\lambda}_{\lfloor s\rfloor}(\mathbf{x}) - Y^{\lambda}_s(\mathbf{x})|^2 \rmd s \nonumber\\
\nonumber
&= \frac{\beta K_1^2}{4} \sum_{j=(k-1)T}^{kT-1} \int_j^{j+1} \E | -\lambda H( Y_j^{\lambda}(\mathbf{x}),x_j) (s-j) +
\sqrt{2 \lambda/\beta} ( \tilde{B}^\lambda_{s} - \tilde{B}^\lambda_{j})|^2 \rmd s  \\
&= \frac{\beta K_1^2}{4} \sum_{j=(k-1)T}^{kT-1}
\left\{ (1/3) \lambda^2\E| H( Y_j^{\lambda}(\mathbf{x}),x_j) |^2 + \fraca{d\lambda}{\beta}\right\} \nonumber\\
&\le \frac{\beta K_1^2}{4} \sum_{j=(k-1)T}^{kT-1} \left\{ \lambda^2 \left[(H^*)^2 + K_1^2 \E |Y_j^{\lambda}(\mathbf{x})|^2  +
K_2^2 | x_j |^2 \right] + \fraca{d\lambda}{\beta}\right\} \nonumber \\
&\le  \bar{C}^{0}(\lambda,\theta_0)  (1-a\lambda)^{(k-1)T}+   \bar{C}^1_k(\bx,\lambda)
\end{align}
where, due to \eqref{eq:moment_SGLD_2}, $\bar{C}^{0}(\lambda,\theta_0)= \lambda \beta K_1^4/(4a)  \E|\theta_0|^2$ and
\begin{multline}
\label{eq:definition-bar-C-k}
\bar{C}^1_k(\bx,\lambda) =   K_1^2 \{1+ \lambda\beta (H^*)^2 + \lambda\beta \cmomentone K_1^2\}/4 + (\lambda^2\beta   K_1^2K_2^2/4) \sum\nolimits_{j=(k-1)T}^{kT-1} |x_j|^2 \\
+ (\lambda^3 \beta K_1^4 c_0/4) \sum\nolimits_{j=(k-1)T}^{kT-1} \sum\nolimits_{l=0}^{j-1} (1-a \lambda)^l |x_{j-1-l}|^2
\end{multline}
where in the case of $k=1$ and $j=0$ the last sum is meant to be $0$.
Moreover, one calculates the bound for $B_k$. Using Lemma~\ref{lem:moment_SGLD_2p} yields that
\begin{align}
\label{eq:bound-moment-4-1}
\E[V_4(Y^{\lambda}_{kT}(\mathbf{x}))] 	 & \leq 2+ 2 (1-a\lambda)^{kT} \E|\theta_0|^{4} +2 \lambda a M(2,d)
 \sum_{j=1}^{kT-1} \left(1-a\lambda\right)^{j-1}|x_{kT-1-j}|^{4}  + 2 \widehat{M}(2,d) \nonumber \\ & =
 2 (1-a\lambda)^{kT} \E|\theta_0|^{4} + D_{k}(\bx,\lambda),
\end{align}
where
\begin{equation}
\label{D_k}
D_{k}(\bx,\lambda) : = 2 \lambda a M(2,d)
 \sum_{j=1}^{kT-1} \left(1-a\lambda\right)^{j-1}|x_{kT-1-j}|^{4}  + 2 \widehat{M}(2,d) + 2.
\end{equation}
Similarly, one obtains, due to Lemma~\ref{lem:moment_SGLD_2p} and Corollary~\ref{aux_proc_conts_fourth_and_Y},
\begin{equation}
\label{eq:bound-moment-4-2}
\E[V_4(\tzeta{\lambda}{kT}{(k-1)T}{\mathbf{x}}{Y^{\lambda}_{(k-1)T}(\mathbf{x})})] 	
\leq 2\rme^{-a}(1-a \lambda)^{(k-1)T} \E|\theta_0|^{4} + \rme^{-a}D_{k-1}(\bx,\lambda) + 3 \mathrm{v}_4(\overline{M}(4))
\end{equation}
By observing \eqref{w_{1,2}_and_KL}, \eqref{KL_estimate}, \eqref{eq:bound-moment-4-1} and \eqref{eq:bound-moment-4-2}, it follows that, for $k=1,\ldots,n$,

\begin{align} \label{semifinal_w_{1,2}_estimate1}
w_{1,2}(\mathcal{L}(Y^{\lambda}_{kT}(\mathbf{x})),
\mathcal{L}(\tzeta{\lambda}{kT}{(k-1)T}{\mathbf{x}}{Y^{\lambda}_{(k-1)T}(\mathbf{x})}) &  \leq \sqrt{\lambda} \left\{(1-a \lambda)^{(k-1)T} \hat{C}^{0}(\lambda,\theta_0) + \hat{C}^1_k(\bx,\lambda) \right\},
\end{align}
where $\hat{C}^0(\lambda,\theta_0)= \bar{C}^{0}(\lambda,\theta_0) + 12 \E|\theta_0|^4$ and
\begin{equation}
\label{eq:definition-hat-C-k}
\hat{C}_k^1(\bx,\lambda)= \bar{C}_k^{1}(\bx, \lambda) +  3  + 3D_{k}(\bx,\lambda) + 3D_{k-1}(\bx,\lambda)+ 9\mathrm{v}_4(\overline{M}(4)) .
\end{equation}
In a similar manner as above, see \eqref{w_{1,2}_and_KL}, one estimates, for any $t\in (nT,(n+1)T]$,
\begin{align} \label{w_{1,2}_and_KL at t}
 w_{1,2}(\mathcal{L}(\tzeta{\lambda}{t}{nT}{\mathbf{x}}{Y^{\lambda}_{nT}(\mathbf{x})}),
\mathcal{L}(Y^{\lambda}_t(\mathbf{x})))  & \le  \sqrt{\frac{\lambda}{2}} \{  (1-a\lambda)^{nT} \bar{C}^{0}(\lambda,\theta_0)+   \bar{C}^1_{n+1}(\bx,\lambda) + B\}
\end{align}
where
\begin{align}
\label{eq:definition-B}
B &\eqdef  \{1 + \E^{1/2}[V_4(Y^{\lambda}_t(\mathbf{x})] + \E^{1/2}[V_4(\tzeta{\lambda}{t}{nT}{\mathbf{x}}{Y^{\lambda}_{nT}(\mathbf{x})})] \}^2.
\end{align}
Thus, due to  Lemmas~\ref{lem:aux_proc_conts_V_p} and \ref{lem:moment_SGLD_2p} and Corollary~\ref{aux_proc_conts_fourth_and_Y},
\begin{equation}
\label{eq:bound-moment-4-2 at n}
\E[V_4(\tzeta{\lambda}{t}{nT}{\mathbf{x}}{Y^{\lambda}_{(k-1)T}(\mathbf{x})})] 	 \leq 2(1-a \lambda)^{nT} \E|\theta_0|^{4} + D_{n}(\bx,\lambda) + 3 \mathrm{v}_4(\overline{M}(4))
\end{equation}
and, analogously, due to equation \eqref{eq:moment_SGLD_2p}, for any $t\in (m, m+1] \subset (nT,(n+1)T]$, where $m$ is a positive integer, the following holds
\begin{align}
\label{eq:bound-moment-4-1 at t}
\E[V_4(Y^{\lambda}_{t}(\mathbf{x}))] \leq & 2 + 2(1-a\lambda(t-m))(1-a\lambda)^{m} \E|\theta_0|^{4} + 2\widehat{M}(2,d)
\nonumber \\ & +2\lambda a M(2,d) \left\{|x_m|^{4} + (1-a\lambda(t-m)) \sum\nolimits_{j=1}^{m} \left(1-a\lambda\right)^{j-1}|x_{m-j}|^{4}  \right\} \nonumber\\	
  \leq & 2 (1-a\lambda)^{nT} \E|\theta_0|^{4} + D_{t, T}(\bx,\lambda),
\end{align}
where
\[
D_{t, T}:= 2+2\lambda a M(2,d) \left\{|x_m|^{4} +  \sum\nolimits_{j=1}^{m} \left(1-a\lambda\right)^{j-1}|x_{m-j}|^{4}  \right\} + 2\widehat{M}(2,d).
\]
Consequently, equations \eqref{w_{1,2}_and_KL at t}, \eqref{eq:definition-B}, \eqref{eq:bound-moment-4-2 at n} and \eqref{eq:bound-moment-4-1 at t}, yield that
\begin{align} \label{semifinal_w_{1,2}_estimate2}
 w_{1,2}(\mathcal{L}(\tzeta{\lambda}{t}{nT}{\mathbf{x}}{Y^{\lambda}_{nT}(\mathbf{x})}),
\mathcal{L}(Y^{\lambda}_t(\mathbf{x})))  & \le\sqrt{\lambda}  \left\{(1-a \lambda)^{nT} \hat{C}^{0}(\lambda,\theta_0) + \hat{C}^1_{t,T}(\bx,\lambda) \right\},
\end{align}
where
\begin{equation}
\label{eq:definition-hat-C-{t,T}}
\hat{C}_{t,T}^1(\bx,\lambda)= \bar{C}_{n+1}^{1}(\bx, \lambda) +  3  + 3D_{n}(\bx,\lambda) + 3D_{t,T}(\bx,\lambda)+ 9\mathrm{v}_4(\overline{M}(4)) .
\end{equation}
Finally, equations \eqref{W_1-estimation}, \eqref{semifinal_w_{1,2}_estimate1} and \eqref{semifinal_w_{1,2}_estimate2} yield that
\begin{align*}
W_1(\mathcal{L}(\tilde{Y}^{\lambda}_t(\mathbf{x})), \mathcal{L}(Y^{\lambda}_t(\mathbf{x})))  \leq &
\sqrt{\lambda} \left(  C_9 \sum_{k=1}^n \rme^{-C_8(n-k)} \left[(1-a \lambda)^{(k-1)T}
\hat{C}^{0}(\lambda,\theta_0) + \hat{C}^1_k(\bx,\lambda)\right]\right) \\ + &
\sqrt{\lambda}  \left\{(1-a \lambda)^{nT} \hat{C}^{0}(\lambda,\theta_0) +
\hat{C}^1_{t,T}(\bx,\lambda) \right\}  \\  \le &
\sqrt{\lambda} \rme^{-\min(C_8,a) n}(n+1) C^{\sharp} \E[ |\theta_0|^2(1+|\theta_0|^2)]  + \sqrt{\lambda} C^{\flat}(\mathbf{x},n,\lambda),
\end{align*}
where
\begin{equation*}
 C^{\sharp}:=  (C_9+1) \left(\lambda_{\mathrm{max}} \beta K_1^4/(4a)+12\right),
\end{equation*}
and
\begin{equation}
\label{eq:definition-C23_1}
 C^{\flat}(\mathbf{x},n,\lambda):=  C_9 \sum_{k=1}^n  e^{-C_8(n-k)}\hat{C}^1_k(\bx,\lambda) +
 \hat{C}^1_{t,T}(\bx,\lambda).
\end{equation}
Notice that $\E[|\theta_{0}|^{2}]\leq{}
\E[|\theta_{0}|^{4}]+1$. Furthermore, for each $y\geq 0$ and $\alpha>0$, $e^{-\alpha y}(y+1)\leq 1+1/\alpha$.
Applying this latter observation with $\alpha=\min(C_8,a)/2$ and $y=n$, it follows that
\begin{eqnarray*}
& & \rme^{-\min(C_8,a) n}(n+1) C^{\sharp} \E[ |\theta_0|^2(1+|\theta_0|^2)]  + C^{\flat}(\mathbf{x},n,\lambda)\\
&\leq& \rme^{-\min(C_8,a) n/2}\left[1+\frac{2}{\min(C_{8},a)}\right]C^{\sharp}(2\E[|\theta_{0}|^{4}]+1)+C^{\flat}(\mathbf{x},n,\lambda)
\end{eqnarray*}
so we can set
\begin{equation}\label{eq:definition-C23_0}
C_{17}:=2\left[1+\frac{2}{\min(C_{8},a)}\right]C^{\sharp},
\end{equation}
and
\begin{equation}\label{eq:definition-C23_1a}
C_{18}(\mathbf{x},n,\lambda):=
C^{\flat}(\mathbf{x},n,\lambda)+\left[1+\frac{2}{\min(C_{8},a)}\right]C^{\sharp}.
\end{equation}
\end{proof}


Recall that $\mathcal{P}(\rset^{q})$ is the set of
	probability measures on $\mathcal{B}(\rset^{q})$ equipped with topology of weak convergence.
	It is known  that $\mathcal{P}(\rset^q)$ can be equipped with the structure
	of a complete separable metric space such that the generated topology coincides with  the topology of weak convergence.
Let us denote by $\Rset:=(\rset^m)^{\nset}$ and by $\Rsigma$ the Borel $\sigma$-algebra associated to the product topology
on $\Rset$.

\begin{lemma}\label{far} Let \Cref{assum:lip} and \Cref{imit} be in force.
The mappings $\tilde{\mu}:\mathbf{x}\to \mathcal{L}(\tilde{Y}_t^{\lambda}(\mathbf{x}))$ and
$\mu:\mathbf{x}\to \mathcal{L}({Y}_t^{\lambda}(\mathbf{x}))$
$\Rsigma/\mathcal{B}(\mathcal{P}(\rset^d))$-measurable for all $0<\lambda$.
\end{lemma}
\begin{proof} Recall that if a sequence $\mathbf{x}^n\in\Rset$ converges to $\mathbf{x}\in\Rset$ in the
product topology, $n\to\infty$ then $\mathbf{x}^n_i\to\mathbf{x}_i$ for each coordinate $i\in\nset$.
We show below, by induction on $j\in \nset$ that
\begin{equation}\label{tritial}
{Y}_t^{\lambda}(\mathbf{x}^n)\to {Y}_t^{\lambda}(\mathbf{x})
\end{equation}
for all $t\in (j,j+1]$ almost surely, $n\to\infty$.
Note that \eqref{tritial} is trivial for $t=0$.

Now notice that
$$
{Y}_t^{\lambda}(\mathbf{x}^n)=\lambda(t-j)H({Y}_j^{\lambda}(\mathbf{x}^n),\mathbf{x}^n_j)+\sqrt{2\lambda}[\tilde{B}^{\lambda}_t-\tilde{B}^{\lambda}_j],
$$
so this tends a.s.\ to ${Y}_t^{\lambda}(\mathbf{x})$ as $n\to\infty$, by continuity of $H(\cdot,\cdot)$ and by the induction hypothesis.
Since almost sure convergence entails convergence in law,
this shows that $\mu$ is, in fact, a continuous functional of $\mathbf{x}$.

Now we turn our attention to $\tilde{\mu}$. For each $\mathbf{x}\in \Rset$, we define a recursive (Picard-type) iteration:
$$
D^0_s(\mathbf{x}):=\theta_0,\ 0\leq s\leq t,\ D^{k+1}_s(\mathbf{x}):=\theta_0+\lambda\int_0^s H(D^k_u(\mathbf{x}),\mathbf{x}_{\lfloor u\rfloor})\, du+
\sqrt{2\lambda}\tilde{B}^{\lambda}_s,\ k\in\nset.
$$
Define $\Phi_k(\mathbf{x}):=\mathcal{L}(D_t^k(\mathbf{x}))$, $\mathbf{x}\in \Rset$, $k\in\nset$.

We now establish for each $k\in\nset$ that, when $\mathbf{x}^n\to\mathbf{x}$, $n\to\infty$,
we have $D^{k}_s(\mathbf{x}^n)\to D^{k}_s(\mathbf{x})$ in $L^1$ (hence also in law). We  check by induction on $k$ that
$$
\sup_{0\leq s\leq t}\E|D^{k}_s(\mathbf{x}^n)-D^{k}_s(\mathbf{x})|\to 0,
$$
which is slightly more (but it is needed for the induction to work). The case $k=0$ is trivial.
Otherwise, using Lipschitz-continuity of $H(\cdot,\cdot)$, for any $s \in \ccint{0,T}$,
\begin{align*}
& \E|D^{k+1}_s(\mathbf{x}^n)-D^{k+1}_s(\mathbf{x})|\\
&\leq \lambda \int_0^s \E|H(D^k_u(\mathbf{x}^n),\mathbf{x}^n_{\lfloor u\rfloor})-H(D^k_u(\mathbf{x}),\mathbf{x}_{\lfloor u\rfloor})|\, \rmd u\\
&\leq \lambda \int_0^s \left\{ \E|H(D^k_u(\mathbf{x}^n),\mathbf{x}^n_{\lfloor u\rfloor})-H(D^k_u(\mathbf{x}),\mathbf{x}^n_{\lfloor u\rfloor})|
+\E|H(D^k_u(\mathbf{x}),\mathbf{x}^n_{\lfloor u\rfloor})-H(D^k_u(\mathbf{x}),\mathbf{x}_{\lfloor u\rfloor})| \right\}\, \rmd u\\
&\leq \lambda\int_0^t \left\{ K_1 \E|D^k_u(\mathbf{x}^n)-D^k_u(\mathbf{x})|+ K_2 \max_{0\leq i\leq \lfloor t\rfloor}|\mathbf{x}^n_i-\mathbf{x}^n_i|\right\} \, \rmd u.
\end{align*}
It follows that
\begin{equation*}
\sup_{0\leq s\leq t}\E|D^{k+1}_s(\mathbf{x}^n)-D^{k+1}_s(\mathbf{x})| \leq \lambda t \left\{ K_1 \sup_{0\leq s\leq t}\E|D^k_s(\mathbf{x}^n)-D^k_s(\mathbf{x})|+ K_2 \max_{0\leq i\leq \lfloor t\rfloor}|\mathbf{x}^n_i-\mathbf{x}^n_i|\right\},
\end{equation*}
which tends to $0$ as $n\to\infty$ by the induction hypothesis and the definition of the convergence in $\Rset$.
We deduce that, for each $k$, the functional $\Phi_k:\mathcal{R}\to\mathcal{P}$ is continuous on $\mathcal{R}$.

Noting $\theta_0\in L^2$, it is well-known (see e.g.\  \cite[Theorem~6.2.2]{arnold:1974}) that $D^k_t(\mathbf{x})\to \tilde{Y}^{\lambda}_t(\mathbf{x})$, $k\to\infty$
in $L^2$. This implies $\Phi_k(\mathbf{x})\to \mathcal{L}(\tilde{Y}^{\lambda}_t(\mathbf{x}))$ in law, for each $\mathbf{x}\in\mathcal{R}$, which shows that the functional $\tilde{\mu}$ is
measurable, being a pointwise limit of continuous functionals.
The proof is complete.
\end{proof}

\begin{lemma}\label{lagel}
Let $(\Uset,\Usigma)$ be a measurable space and let the mappings $\mu:\Uset\to\mathcal{P}(\rset^d)$,
$\tilde{\mu}:\Uset\to\mathcal{P}(\rset^d)$ be $\Usigma/\mathcal{B}(\mathcal{P}(\rset^d))$-measurable. Let $\zeta$ be a probability
law on $\Usigma$.
If $W_1(\tilde{\mu}(u),\mu(u))\leq \kappa(u)$ holds for every $u\in\Uset$
where $\kappa:\Uset\to [0,1]$ is a measurable function then
\[
W_1\left(\int_{\Uset}\tilde{\mu}(u)\, \zeta(\rmd u),\int_{\Uset}{\mu}(u)\, \zeta(\rmd u)\right)\leq \int_{\Uset}\kappa(u)\, \zeta(\rmd u).
\]
\end{lemma}
\begin{proof}
By \cite[Corollary~5.22]{VillaniTransport}, there is a measurable choice $u \to \pi(u)$ such that for
each $u$, $\pi(u)$ is a $W_1$-optimal transference plan between $\mu(u)$ and $\tilde{\mu}(u)$. For any $A \in \rset^d$,
$\int_{\Uset} \zeta(\rmd u) \pi(u)(A \times \rset^d) = \int_{\rset^d} \zeta(\rmd u) \mu(u)(A) $ and
$\int_{\Uset} \zeta(\rmd u) \pi(u)(\rset^d \times A) = \int_{\rset^d} \zeta(\rmd u) \tilde{\mu}(u)(A) $. Therefore
\begin{align*}
W_1\left(\int_{\Uset}\tilde{\mu}(u)\, \zeta(\rmd u),\int_{\Uset}{\mu}(u)\, \zeta(\rmd u)\right)
\leq  \int_{\Uset} \zeta(\rmd u)  \int_{\rset^{2d}} \pi(u) (\rmd x \rmd y) |x-y|.
\end{align*}
The proof follows since $\int_{\rset^{2d}} \pi(u) (\rmd x \rmd y) |x-y|  = W_1(\mu(u),\tilde{\mu}(u)) \leq \kappa(u)$.
\end{proof}

\begin{corollary}\label{crux}
For each $0<\lambda\leq \lambda_{\max}$ and  $t\in (nT,(n+1)T]$, we get
\[
W_1(\mathcal{L}(\tilde{Y}^{\lambda}_t(\mathbf{X})),
\mathcal{L}(Y^{\lambda}_t(\mathbf{X})))\leq
\lambda^{1/2}[\rme^{-\min(C_8,a) n/2} C_{17}E[|\theta_{0}|^{4}]+C_{19}],
\]
where $C_{19}:=\sup_{\lambda\leq\lambda_{\mathrm{max}}}\sup_{n\in\nset}E[C_{18}(\mathbf{X},n,\lambda)]<\infty$.
\end{corollary}
\begin{proof} Recall first that as $X$ is conditionally $L$-mixing,
$A:=1+\sup_{n\in\nset}\E[|X_n|^4]<\infty$. Fix $n$ such that $n <t\leq n+1$.
Denote by $\zeta$  the law of $\mathbf{X}$.
Define
\[
\tilde{\mu}(\mathbf{x}):=\mathcal{L}(\tilde{Y}_t^{\lambda}(\mathbf{x})),\quad
\mu(\mathbf{x}):=\mathcal{L}({Y}_t^{\lambda}(\mathbf{x})).
\]
Lemma \ref{far} implies the measurability of these functionals. Let
\[
\kappa(\mathbf{x},t):=\lambda^{1/2}(\rme^{-\min(C_8,a) n/2} C_{17} \E[ |\theta_0|^4]+  C_{18}(\mathbf{x},n,\lambda)),
\]
for each $\mathbf{x}\in\mathcal{R}$, where $C_{18}(\mathbf{x},n,\lambda)$ is given in \eqref{eq:definition-C23_1a}.
Now the statement follows by Lemma \ref{lagel} provided that we show
$C_{19}< \infty$.
By the definitions of $\hat{C}^1_k(\bx,\lambda)$ and $\hat{C}^{1}_{t,T}(\bx,\lambda)$ this boils down to showing
that $\sup_{\lambda\leq\lambda_{\mathrm{max}}}\sup_{k}\E[S_{1}(\lambda,k)+S_{2}(\lambda,k)]<\infty$, where
\begin{eqnarray*}
S_{1}(\lambda,k)&=&\lambda^3 \sum_{j=(k-1)T}^{kT-1}  \sum_{l=0}^{j}(1-a\lambda)^l E|X_{j-l}|^2 +
\lambda^2\sum_{l=(k-1)T}^{kT-1} E|X_l|^2\\
S_{2}(\lambda,k) &=& \lambda \sum_{j=0}^{kT} \left(1-a\lambda\right)^jE|X_{(k-1)T-j}|^{4}.
\end{eqnarray*}
This is clear since
\begin{equation*}
\E[S_{1}(\lambda,k)]\leq \lambda^{3}\frac{A}{a\lambda}\frac{1}{\lambda}+\lambda^{2}\frac{A}{\lambda}\leq A\lambda_{\mathrm{max}}
\left(1+\frac{1}{a}\right),
\end{equation*}
and
\begin{equation*}
\E[S_{2}(\lambda,k)]\leq \lambda \frac{A}{a\lambda}\leq \frac{A}{a}.
\end{equation*}
\end{proof}

\begin{lemma} \label{contractionconst} The contraction constant in Proposition \ref{prop:contra} is given by
$$
C_{8}=\min\{\bar{\phi}, C_6(p), 4C_7(p)\epsilon C_6(p)\}/2,
$$
where the explicit expressions for $C_6(p)$ and $C_7(p)$ can be found in Lemma~\ref{lyapp} and  $\phi$ is given by
\[
\bar{\phi}= \left(\sqrt{\fraca{4\pi}{K_1}} \bar{b}  \exp\left(\left(\bar{b} \, \fraca{\sqrt{K_1}}{2} +\fraca{2}{\sqrt{K_1}}\right)^2\right) \right)^{-1} \,.
\]
Furthermore, any $\epsilon$ can be chosen which satisfies the following inequality
\[
\epsilon  \leq 1 \wedge \left(8C_7(p) \sqrt{\fraca{\pi}{K_1}}\int_0^{\tilde{b}}\exp\left(\left(s\fraca{\sqrt{K_1}}{2}+\fraca{2}{\sqrt{K_1}}\right)^2\right) \,\rmd s\right)^{-1},
\]
where $\tilde{b}=\sqrt{2C_7(p)/C_6(p)-1}$, $\bar{b} = \sqrt{4C_7(p)(1+C_6(p))/C_6(p)-1}$.
The constant $C_9$ is given as the ratio $C_{11}/C_{10}$, where $C_{11},\,C_{10}$ are given explicitly in the proof below.
\end{lemma}
\begin{proof}
Consider the Lyapunov function $V_p(\theta) = (|\theta|^2+1)^{p/2}$, for any $\theta \in \rset^d$ and $p \geq 2$. Notice that $\nabla V_p(\theta) = p\theta(|\theta|^2+1)^{p/2-1}$.
As in \cite{eberle:guillin:zimmer:Trans:2019}, define a bounded non-decreasing function: $Q(\epsilon): (0,\infty) \rightarrow \mathbb{R}_{+}$ by
\[
Q(\epsilon) = \sup \frac{|\nabla V_p|}{\max\{V_p, 1/\epsilon\}}.
\]
For calculating the constants we need an estimate for $Q(\epsilon)$.{}
We consider the following three cases:
\begin{enumerate}
\item Consider $\epsilon \in (0,2^{-p/2})$. For $|\theta| < \sqrt{(1/\epsilon)^{2/p}-1}$, we have $V_p(\theta) < 1/\epsilon$, and
\[
Q(\epsilon) =\sup_{|\theta| < \sqrt{(1/\epsilon)^{2/p}-1}} \epsilon  p|\theta|(|\theta|^2+1)^{p/2-1} = \epsilon^{2/p} p \sqrt{(1/\epsilon)^{2/p}-1}.
\]
On the other hand, for $|\theta|  \geq \sqrt{(1/\epsilon)^{2/p}-1}$, $V_p(\theta) \geq 1/\epsilon$, and
\[
Q(\epsilon) =\sup_{} \frac{p|\theta|}{|\theta|^2+1} = \epsilon^{2/p} p \sqrt{(1/\epsilon)^{2/p}-1},
\]
since for $\epsilon \in (0,2^{-p/2})$, $|\theta| >1$. Therefore, $Q(\epsilon)=\epsilon^{2/p} p \sqrt{(1/\epsilon)^{2/p}-1}\leq p/2$
for all $\epsilon \in (0,2^{-p/2})$.

\item  For the second case, consider $\epsilon \in (2^{-p/2}, 1)$. Then,
by using the same arguments as above, one obtains for $|\theta| < \sqrt{(1/\epsilon)^{2/p}-1}$,
$Q(\epsilon)  = \epsilon^{2/p} p \sqrt{(1/\epsilon)^{2/p}-1} $, while for $|\theta|  \geq \sqrt{(1/\epsilon)^{2/p}-1}$, $Q(\epsilon)  = p/2$.
Thus, one obtains $Q(\epsilon)\leq  p/2 $ for all $\epsilon \in (2^{-p/2}, 1)$.

\item Finally, for $\epsilon \geq 1$, we have $Q(\epsilon) =  p/2 $, since $V_p(\theta) \geq 1$ for all $\theta \in \rset^d$.
\end{enumerate}

In the first two cases above, we used the fact that $p/2 \geq \epsilon^{2/p} p \sqrt{(1/\epsilon)^{2/p}-1}$ for all $\epsilon \in (0,1)$.
Indeed, this is true since, squaring both sides, we have
\[
1 \geq 4\epsilon^{4/p}((1/\epsilon)^{2/p}-1) \iff 4\epsilon^{4/p} -4\epsilon^{2/p}+1 \geq 0 \iff (2\epsilon^{2/p}-1)^2 \geq 0.
\]
{}
Combining all the three cases, one obtains $Q(\epsilon)  \leq p/2$ for all $\epsilon >0$.

In \cite{eberle:guillin:zimmer:Trans:2019} a further key qunatity is $R_{2}\geq 0$.
We note that, by its definition in Section 2 of \cite{eberle:guillin:zimmer:Trans:2019},
it satisfies
\begin{eqnarray*}
& & R_2\leq 2 \sup \{|\theta|: V_p(\theta) \leq 4C_7(p)(1+C_6(p))/C_6(p) \} \\
&\implies& R_2 \leq
\overline{R}_{2}:=2\sqrt{(4C_7(p)(1+C_6(p))/C_6(p))^{2/p}-1}
\end{eqnarray*}
as well as
$$
R_2 \geq
\underline{R}_{2}:=\sqrt{(4C_7(p)(1+C_6(p))/C_6(p)-1)^{2/p}-1}.
$$

We now check the requirements of \cite[Theorem~2.2]{eberle:guillin:zimmer:Trans:2019} for $\epsilon$. It is required that
\[
(4C_7(p)\epsilon)^{-1} \geq \int_0^{R_1} \int_0^s \exp \left(\frac{1}{2}\int_r^s u \kappa(u)\, du+ 2Q(\epsilon)(s-r)\right)\,dr\,ds,
\]
where in our case $\kappa(u) = K_1$ and $Q(\epsilon)\leq p/2$. Hence $\epsilon$ is suitable whenever
\begin{align*}
(4C_7(p)\epsilon)^{-1}	& \geq \int_0^{R_1} \int_0^s \exp \left(\frac{1}{2}\int_r^s K_1 u\, du+ p(s-r)\right)\,dr\,ds\\
						& =  \int_0^{R_1} \int_0^s \exp \left(\frac{K_1}{4}(s^2-r^2)+ p(s-r)\right)\,dr\,ds\\
						& =  \int_0^{R_1} \exp\left(\left(\frac{\sqrt{K_1}}{2}s+\frac{p}{\sqrt{K_1}}\right)^2\right) \int_0^s \exp \left(-\left(\frac{\sqrt{K_1}}{2}r+\frac{p}{\sqrt{K_1}}\right)^2\right)\,dr\,ds,
\end{align*}
which implies by setting $v/\sqrt{2}=\sqrt{K_1}r/2 +p/\sqrt{K_1}  $, ($dv = \sqrt{K_1/2} dr$)
\begin{align*}
(4C_7(p)\epsilon)^{-1}	& \geq  \sqrt{\frac{2}{K_1}}\int_0^{R_1} \exp\left(\left(\frac{\sqrt{K_1}}{2}s+\frac{p}{\sqrt{K_1}}\right)^2\right) \int_{p\sqrt{2/K_1}}^{\sqrt{K_1/2}s+p\sqrt{2/K_1}} \exp \left(-\frac{v^2}{2}\right)\,dv\,ds\\
						& = \sqrt{\frac{4\pi}{K_1}}\int_0^{\tilde{b}} \exp\left(\left(\frac{\sqrt{K_1}}{2}s+\frac{p}{\sqrt{K_1}}\right)^2\right) \left(\Phi\left(\sqrt{K_1/2}s+p\sqrt{2/K_1}\right)-\Phi\left(p\sqrt{2/K_1}\right)\right)\,ds,
\end{align*}
where $\tilde{b}=\sqrt{(2C_7(p)/C_6(p))^{2/p}-1} >0$ and $\Phi(\cdot)$ is the cumulative distribution function of
the standard normal distribution.

The inrements of a cumulative distribution function can be at most one.
To ease the calculations of $C_{10}$ and $C_{11}$ below, it is thus enough for $\epsilon$ to satisfy the following inequality:
\[
 \epsilon  \leq  1\wedge \left(8C_7(p) \sqrt{\frac{\pi}{K_1}}\int_0^{\tilde{b}}\exp\left(\left(\frac{\sqrt{K_1}}{2}s+\frac{p}{\sqrt{K_1}}\right)^2\right) \,\rmd s\right)^{-1}.
\]

In \cite{eberle:guillin:zimmer:Trans:2019} a further key quantity is $\beta$, which we denote by $\phi$ in order to avoid
a clash of notation. We calculate $\phi$ using its definition in Theorem 2.2 of \cite{eberle:guillin:zimmer:Trans:2019},
noting that $Q(\epsilon)\leq p/2$,
\begin{align*}
\phi^{-1}				& =\int_0^{R_2}  \int_0^s \exp \left(\frac{1}{2}\int_r^s K_1 u\, \rmd u+ 2Q(\epsilon)(s-r)\right)\, \rmd r\, \rmd s\\
& \leq\int_0^{R_2}  \int_0^s \exp \left(\frac{1}{2}\int_r^s K_1 u\, \rmd u+ p(s-r)\right)\, \rmd r\, \rmd s\\
						& = \sqrt{\frac{4\pi}{K_1}}\int_0^{\bar{b} } \exp\left(\left(\frac{\sqrt{K_1}}{2}s+\frac{p}{\sqrt{K_1}}\right)^2\right) \left(\Phi\left(\sqrt{K_1/2}s+p\sqrt{2/K_1}\right)-\Phi\left(p\sqrt{2/K_1}\right)\right)\,ds,
\end{align*}
where $\bar{b} = \sqrt{(4C_7(p)(1+C_6(p))/C_6(p))^{2/p}-1}>0$. One notices that
\[
\phi \geq \bar{\phi} = \left(\sqrt{\frac{4\pi}{K_1}} \bar{b}  \exp\left(\left(\frac{\sqrt{K_1}}{2}\bar{b} +\frac{2}{\sqrt{K_1}}\right)^2\right) \right)^{-1}
\]
hence Theorem 2.2 of \cite{eberle:guillin:zimmer:Trans:2019} implies that we can choose
\[
C_8 = \bar{C}_8 :=\min\{\bar{\phi}, C_6(p), 4C_7(p)\epsilon C_6(p)\}/2.
\]



As for the calculations of $C_{10}$ and $C_{11}$ in \eqref{lajja}, recall the definition of $\mathcal{W}_{\rho_{2}}$ in
\eqref{lajjja}.
Moreover, recall that $f$, $F$ and $R_2$  are given in  \cite[Section~5]{eberle:guillin:zimmer:Trans:2019} and
satisfy $\frac{1}{2} F(r) \leq f(r) \leq F(r)$ for $r \leq R_2$ and $f(r) = f(R_2)$ for $r \geq R_2$.
In addition, $r\exp{(-K_1R_2^2/4 -pR_2)} \leq F(r) \leq r$ for all $r \leq R_2$ and $f(r) \leq R_2$ for all $r>0$.

With these tools at hand, we take $\theta,\theta'$ such that $|\theta-\theta'|=r \leq R_2$ and estimate
\begin{align*} &
[1 \wedge |\theta-\theta'|](1+V_2(\theta)+ V_2(\theta'))\\
 	& \leq \epsilon^{-1}|\theta-\theta'|(\epsilon+\epsilon V_2(\theta)+\epsilon  V_2(\theta'))\\
											& \leq 2\epsilon^{-1}\exp{(K_1R_2^2/4 +pR_2)} \left(\frac{1}{2}F(|\theta-\theta'|)\right)(1+\epsilon V_2(\theta)+\epsilon  V_2(\theta'))\\
											& \leq \bar{C}_{10}^{-1} f(|\theta-\theta'|)(1+\epsilon V_2(\theta)+\epsilon  V_2(\theta')),
\end{align*}
where $\bar{C}_{10}  = \frac{\epsilon}{2} \exp{(-K_1 \overline{R}_2^2/4 -p\overline{R}_2)}$. For $r > R_2$ we get
\begin{align*}
 f(|\theta-\theta'|)(1+\epsilon V_2(\theta)+\epsilon  V_2(\theta')) 	&=  f(R_2)(1+\epsilon V_2(\theta)+\epsilon  V_2(\theta')) \\
 													&\geq \tilde{C}_{10}[1 \wedge |\theta-\theta'|](1+V_2(\theta)+ V_2(\theta')),
\end{align*}
where $\tilde{C}_{10} =  \frac{\epsilon}{2} \underline{R}_2\exp{(-K_1\overline{R}_2^2/4 -p\overline{R}_2)}$.
We can thus take $C_{10} =  \min\{\bar{C}_{10} , \tilde{C}_{10} \}$.

To calculate $C_{11}$, one considers, for $|\theta-\theta'|=r \leq R_2$
\begin{align*}
 f(|\theta-\theta'|)(1+\epsilon V_2(\theta)+\epsilon  V_2(\theta')) 	& \leq |\theta-\theta'|(1+\epsilon V_2(\theta)+\epsilon  V_2(\theta')) 	 \\
 													& \leq C_{11}[1 \wedge |\theta-\theta'|](1+V_2(\theta)+ V_2(\theta')) ,
\end{align*}
where $C_{11} = 1+R_2$. In the case where $r > R_2$, we also get
\begin{align*}
 f(|\theta-\theta'|)(1+\epsilon V_2(\theta)+\epsilon  V_2(\theta')) 	&  =  f(R_2)(1+\epsilon V_2(\theta)+\epsilon  V_2(\theta')) 	 \\
 													& \leq C_{11}[1 \wedge |\theta-\theta'|](1+V_2(\theta)+ V_2(\theta')),
\end{align*}
hence the choice $C_{11}  = 1+ R_2$ is indeed fine.
\end{proof}

\subsection{Proof of Main Result}

\begin{proof}[Proof of Theorem \ref{main}.]
Trivially, $\E^{3/4}[V_{4}(\theta_{0})]\leq 1+\E[V_{4}(\theta_{0})]$ and $\E[V_{4}(\theta_{0})]\leq 2+2\E[|\theta_{0}|^{4}]$.
We estimate, for $kT\leq t\leq (k+1)T$,
\begin{align*} &
W_1(\mathcal{L}(Y^{\lambda}_t(\mathbf{X}),\pi_{\beta}) \\
&\leq W_1(\mathcal{L}(Y^{\lambda}_t(\mathbf{X}),\mathcal{L}(\tilde{Y}^{\lambda}_t(\mathbf{X}))
+ W_1(\mathcal{L}(\tilde{Y}^{\lambda}_t(\mathbf{X})),\mathcal{L}(L^{\lambda}_t))+
W_1(\mathcal{L}(L^{\lambda}_t),\pi_{\beta}) \\
&\leq
\lambda^{1/2}\left[\rme^{-\min\{C_{8},a/4\}k/2}[C_{19}\E[V_{4}(\theta_{0})]+C_{15}\E^{3/4}[V_{4}(\theta_{0})]]+C_{17}+C_{15}\right]
+{C_9} \rme^{-C_8\lambda t}w_{1,2}(\theta_0,\pi_{\beta})\\
&\leq \rme^{-\min\{C_{8},a/4\}k/2}\lambda^{1/2}(C_{19}+C_{15})[\E[V_{4}(\theta_{0})]+1]\\
&+ (C_{17}+C_{15})\sqrt{\lambda}+C_{9}\rme^{-\min\{C_{8},a/4\}k/2}
[1+E[V_{2}(\theta_{0})]+\int_{\mathbb{R}^{d}}V_{2}(\theta)\pi_{\beta}(\mathrm{d}\theta)]\\
&\leq \rme^{-\min\{C_{8},a/4\}k/2}\lambda^{1/2}(C_{19}+C_{15})[2\E[|\theta_{0}|^{4}]+3]\\
&+ (C_{17}+C_{15})\sqrt{\lambda}+C_{9}\rme^{-\min\{C_{8},a/4\}k/2}
[2+E[V_{4}(\theta_{0})]+\int_{\mathbb{R}^{d}}V_{2}(\theta)\pi_{\beta}(\mathrm{d}\theta)]\\
&\leq \rme^{-\min\{C_{8},a/4\}k/2}\lambda^{1/2}_{\mathrm{max}}(C_{19}+C_{15})[2\E[|\theta_{0}|^{4}]+3]\\
&+ (C_{17}+C_{15})\sqrt{\lambda}+C_{9}\rme^{-\min\{C_{8},a/4\}k/2}
[4+2\E[|\theta_{0}|^{4}]+\int_{\mathbb{R}^{d}}V_{2}(\theta)\pi_{\beta}(\mathrm{d}\theta)]\\
&\leq \rme^{-\min\{C_{8},a/4\}k/2}[2\lambda^{1/2}_{\mathrm{max}}(C_{19}+C_{15})+2C_{9}] E[|\theta_{0}|^{4}]\\
&+ \rme^{-\min\{C_{8},a/4\}k/2}[3(C_{19}+C_{15})\lambda^{1/2}_{\mathrm{max}}+4C_{9}+C_{9}\int_{\mathbb{R}^{d}}V_{2}(\theta)\pi_{\beta}(\mathrm{d}\theta)]\\
&+ (C_{17}+C_{15})\sqrt{\lambda},
\end{align*}
by Corollary \ref{crux},  Corollary \ref{cor_vizier}, Proposition \ref{prop:contra} and by \eqref{eq:lucia}.
Noting \eqref{antoniojobim} and $\lfloor n\lambda\rfloor \lfloor 1/\lambda\rfloor\leq n$, this implies, for all $n\in\mathbb{N}$,
\begin{equation*}
W_1(\mathcal{L}(\theta^{\lambda}_n,\pi_{\beta})\leq \rme^{-C_{0}\lfloor n\lambda\rfloor}\bar{C}_{1}[1+\E[|\theta_{0}|^{4}]]+C_{2}\sqrt{\lambda},	
\end{equation*}
where ${C}_{0}=\min\{C_{8},a/4\}/2$,
$$
\bar{C}_{1}=\lambda_{\mathrm{max}}^{1/2}(C_{19}+C_{15})+2C_{9}+3\lambda^{1/2}_{\mathrm{max}}(C_{19}+C_{15})+4C_{9}+
C_{9}\int_{\mathbb{R}^{d}}V_{2}(\theta)\pi_{\beta}(\mathrm{d}\theta)
$$
and $C_{2}=C_{17}+C_{15}$. We can thus set $C_{1}:=e^{C_{0}}\bar{C}_{1}$ and conclude.
\end{proof}

\begin{remark} {\rm The proof of Lemma \ref{contractionconst} shows that $C_9$ has a rather poor (exponential) dependence on the dimension $d$,
see the definitions of $C_{10}$ and $\overline{R}_{2}$ therein as well as the definition of $C_{7}(p)$ in Lemma \ref{lyapp}.
Improvements on the dimension dependence here would require enhancing the coupling arguments of
\cite[Corollary~2.3]{eberle:guillin:zimmer:Trans:2019} significantly.}
\end{remark}

\section{Applications to non-convex optimization}\label{sec:application}

\begin{example}
{\rm Let
$Z_{n}\in\mathbb{R}^{m}$, $n\in\mathbb{Z}$ be a (strict-sense) stationary sequence.
Let us consider the problem of online nonlinear prediction of $Z_{n}$ as a function of the $p$ previous
observations $Z_{n-1},\ldots,Z_{n-p}$. We use a predictor of the form $\hat{Z}_n(\theta)=
f_{\theta}(Z_{n-1},\dots,Z_{n-p})$, where $f_{\theta}:\mathbb{R}^{p\times m}\to{}
\mathbb{R}^{m}$, $\theta\in\mathbb{R}^{d}$ is a parametric family of (non-linear)
twice continuously
differentiable functions,
such as the output of a neural network. We seek to minimize the regularized mean-square error,
that is,
\begin{equation}\label{nonconv-problem}
U(\theta)= \E[ |
Z_p - f_{\theta}(Z_{p-1},\dots,Z_0)|^2] + c |\theta|^2{}
\end{equation}
for some $c>0$. Here
\[
H^i(\theta,z)=2\left \langle z^p - f_{\theta}(z^{p-1},\dots,z^0),
\frac{\partial f_{\theta}(z^{p-1},\dots,z^0)}{\partial \theta^i}\right\rangle_{\mathbb{R}^{m}} + 2c \theta^i
\]
for each $i=1,\ldots,d$. Let $Z$ be conditionally $L$-mixing.
If we assume that $f_{\theta}$, $\partial_{\theta}f_{\theta}$, $\partial_{\theta\theta}f_{\theta}$
as well as $Z_{n}$
are all bounded and $z\to f_{\theta}(z)$ and $z\to \partial_{\theta}f_{\theta}(z)$ are Lipschitz,
then the assumptions of our paper hold, as easily checked.
The SGLD then provides an algorithm to optimize the prediction procedure.

Let us now apply this framework to online price prediction, a procedure of paramount importance for
econometric analysis and algorithmic trading (see e.g.\ the paper \cite{energy}
which surveys 27 methods, including several neural network-based approaches).

Denote by
$Z_{n}\in\mathbb{R}^{m}$ the return vector on $m$ assets at time $t$.
While stationarity
of the process $Z$ holds on appropriate time scales (see Subsection 3.1 of \cite{cont}), independence
badly fails (see Section 5 of \cite{cont} and the references therein).
Conditional $L$-mixing
holds e.g.\ when $Z$ is a (possibly nonlinear) Lipschitz functional of an infinite moving average processes, see \cite{4}.
Another example for conditional $L$-mixing is rough volatility models, see \cite{rough,4}.
Our SGLD algorithm with dependent data can then
be used to find the optimizer of \eqref{nonconv-problem}.

Financial applications provide a rich source of problems where stochastic approximation needs to
be used in settings with dependent data: optimal posting of orders, optimal split of orders, etc.
Here we do not enter into more details, see \cite{laruelle1,laruelle2,laruelle3}.}

\end{example}

\begin{example}
	
{\rm We sketch a general optimization framework here. It is often the case that a deterministic function
we wish to minimize has some representation as the expectation of a functional of a random variable.
It is also often clear that the optimizer necessarily lies in some (big) compact set $\mathrm{B}_{R''}$
where $R''$ can be estimated.

Let $R>0$ and let $\overline{U}:\mathrm{B_{R}}\to \mathbb{R}_{+}$ be a possibly non-convex function. We assume
that it admits a stochastic representation $\overline{U}(\theta)=E[\overline{u}(\theta,X)]$, $\theta\in\mathrm{B}_{R}$
where $\overline{u}:\mathrm{B}_{R'}\times \mathbb{R}^{m}\to\mathbb{R}_{+}$ is continuous and continuously differentiable
in $\mathrm{int}\,\mathrm{B}_{R'}$ for some $R'>R$ and $\frac{\partial}{\partial{\theta}}\overline{u}(\theta,x)${}
is jointly Lipschitz-continuous in $(\theta,x)\in \mathrm{B}_{R}\times\mathbb{R}^{m}$. $X$ is a $\mathbb{R}^{m}$-valued random variable
which we assume bounded, for simplicity.
We assume that $\overline{U}$ has a unique minimizer
$\theta^{*}\in \mathrm{int}\,\mathrm{B}_{R''}$
with some $R''<R$. We may always assume $u^{*}:=\overline{U}(\theta^{*})\geq 0$ (by adding a suitably large
constant).


We propose an approach to find $\theta^{*}$ using SGLD.
The case of multiple global minimizers can be handled similarly.
Since $\theta^{*}\in \mathrm{int}\, \mathrm{B}_{R''}$, one deduces $\inf_{|\theta|=R''}\overline{U}(\theta)\geq u^{*}+\kappa$
for some $\kappa>0$.
Then, continuity implies that for some $\delta>1$ which is close enough to $1$,
$$
\inf_{\lambda\in [1,\delta^{2}],|\theta|=R''} E\left[\overline{u}(\theta\sqrt{\lambda},X)\left(1-\frac{\lambda}{\delta^{2}}\right)\right]+
|\theta\sqrt{\lambda}|^{2}\frac{\lambda}{\delta^{2}}\geq u^{*}+\kappa/2.
$$
One reasonably assumes that $R''\delta<R$ and proceeds by defining, for $\lambda\in [1,\delta^{2}]$ and for all $\theta$ with $|\theta|=R''$,
$u(\theta\sqrt{\lambda},x):=\overline{u}(\theta\sqrt{\lambda},X)\left(1-\frac{\lambda}{\delta^{2}}\right)+|\theta\sqrt{\lambda}|^{2}\frac{\lambda}{\delta^{2}}$.
An alternative way of writing this is
\begin{equation}\label{connecticut}
u(\theta,x)=\overline{u}(\theta,X)\left(1-\frac{|\theta|^{2}}{\delta^{2}(R'')^{2}}\right)+|\theta|^{2}
\frac{|\theta|^{2}}{\delta^{2}(R'')^{2}}
\end{equation}
for $\theta$ with $R''\leq |\theta|\leq R''\delta$.
Furthermore, for $\theta\in\mathrm{int}\,\mathrm{B}_{R''}$ define $u(\theta,x):=\overline{u}(\theta,x)$, $x\in\mathbb{R}^{m}$ and
for $\theta\notin \mathrm{B}_{R''\delta}$ set $u(\theta,x):=|\theta|^{2}$.
Define also $U(\theta):=E[u(\theta,X)]$, $\theta\in\mathbb{R}^{d}$.

It is claimed that $\theta^{*}$ is also the unique minimizer of $U(\theta)$. To show this, it suffices to demonstrate that
$U(\theta)>u^{*}$ holds for all $\theta$ with $|\theta|\geq R''$. This
holds for $R''\leq \theta\leq R''\delta$ by the choice of $\delta$ and it is trivial for $|\theta|\geq R''\delta$
since $U$ is monotone in $|\theta|$ over that set.

It is not difficult to see, using \eqref{connecticut},
that $u(\theta,x)$ is continuously differentiable and $H(\theta,x):=\frac{\partial}{\partial\theta}u(\theta,x)${}
is jointly Lipschitz-continuous. The dissipativity condition is trivial for $H$ as it is obvious
for $\theta\to |\theta|^{2}$ and $H$ coincides with the latter function outside a compact set.
Let $X_{n}$, $n\geq 1$ be a stationary sequence with common law equal to that of $X$, satisfying
Assumption \ref{assum:lmiu}.
One then implements the SGLD algorithm for $\beta$ large, $\lambda$ small. For $n$ large enough,
$\theta_{n}^{\lambda}$ is a good approximate sample from $\pi_{\beta}$ and hence, by maximality of $\theta^{*}$,
a good estimate for $\theta^{*}$ (where goodness of the estimate is quantified by our main result, Theorem \ref{main}).}

\end{example}

\begin{example} {\rm A deep neural network with weight constraints falls under the scope of the present section.
We denote by $d_{1}$ the number of hidden layers, let $d_{2}$ be the number of nodes in each layer.
The parameter space is $\mathbb{R}^{d}$, where $d:=d_{1}\times d_{2}\times d_{2}$.
Elements $\mathbf{w}\in\mathbb{R}^{d}$ are weight matrices where $\mathbf{w}=[w_{k,j,l}]$
and the indices' ranges are $k=0,\ldots,d_{1}-1$, $j,l=1,\ldots,d_{2}$.	

The (training) data sequence $X_{n}$, $n\geq 1$ is $m$-dimensional bounded and stationary,
satisfying Assumption \ref{assum:lmiu}.
For simplicity we assume $m=d_{2}+1$.

A generic element of $\mathbb{R}^{m}$ is denoted by $\mathbf{x}:=(x_{1},\ldots,x_{m})$.
We fix an activation function $\alpha:\mathbb{R}\to\mathbb{R}$ that is twice continuously differentiable.

A possible specification is $$
\alpha(u):=\frac{1}{1+e^{-u}}+1,\ u\in\mathbb{R},
$$
the well-known sigmoid function.

We recursively define a doubly indexed sequence of functions $f_{k,l}:\mathbb{R}^{d}_{+}\times\mathbb{R}^{m}\to\mathbb{R}$ by
$$
f_{0,l}(\mathbf{w},\mathbf{x}):=x_{l},\ 1\leq l\leq d_{2},\ (\mathbf{w},\mathbf{x})\in\mathbb{R}^{d+m},
$$
and
$$
f_{k+1,l}(\mathbf{w},\mathbf{x}):=\alpha\left(\sum_{j=1}^{d_{2}}w_{k,j,l}f_{k,j}(\mathbf{w},\mathbf{x})\right),\ 1\leq l\leq d_{2},
$$
for $k:=0,\ldots,d_{1}-1$. We set
$$
{\overline{u}}(\mathbf{w},\mathbf{x}):=\left(\frac{1}{d_{2}}\sum_{l=1}^{d_{2}}f_{d_{1},l}-x_{d_{2}+1}\right)^{2}
$$
and we aim at minimizing $E[F(\mathbf{w},X_{0})]$ in the parameter $\mathbf{w}$.

We imagine that the last coordinate of the $X_{n}$ are some (noisy) functionals of their first $d_{2}$
coordinates and we try to use a neural network, characterized by the weights $\mathbf{w}$, that mimics
this functional relationship in the best possible way. This would amount to minimizing $E[\overline{u}(\mathbf{w},X_{0})]$
in $\mathbf{w}$.

It is a standard technique against overfitting to set a maximum for the norm of $\mathbf{w}$ when optimizing.
This means maximizing over $\mathrm{B}_{R}$ for some $R$.
 By induction, it is easy to show that $f_{k,l}$ are twice continuously differentiable in $\mathbf{w},\mathbf{x}$ and
so is $\overline{u}$.  This implies joint Lipschitz-continuity of $\frac{\partial}{\partial\mathbf{w}}\overline{u}$
on the compact set $\mathrm{B}_{R}$, for each $R$. It follows that we can apply the optimization procedure
outlined above and obtain reassuring theoretical guarantees for the convergence of SGLD iterates for deep neural
networks with weight constraints.}

\end{example}

\begin{example}
{\rm For a given input vector $x \in \mathbb{R}^{m_1}$, an autoencoder aims to learn a cost-effective representation of $x$. Consider the following neural network:
\[
\widehat{x} = \sigma_2(g_2(W_2)\sigma_1(g_1(W_1)x+b_1)+b_2),
\]
where $W_1 \in \mathbb{R}^{d_1 \times {m_1}}, W_2 \in \mathbb{R}^{{m_1} \times d_1}$ are the weights, $b_1 \in \mathbb{R}^{d_1}, b_2 \in \mathbb{R}^{m_1}$ are the biases, $\sigma_1: \mathbb{R}^{d_1} \rightarrow \mathbb{R}^{d_1}, \sigma_2: \mathbb{R}^{m_1} \rightarrow \mathbb{R}^{m_1}$ are the elementwise activation functions, and moreover, the $(i_1,j_1)$-th element of $g_1: \mathbb{R}^{d_1\times {m_1}}\rightarrow \mathbb{R}^{d_1\times {m_1}}$ is given by $g_1^{(i_1,j_1)}(W_1) = C_1\tanh(W_1^{(i_1,j_1)}/C_1)$ whereas the $(i_2,j_2)$-th element of $g_2: \mathbb{R}^{{m_1} \times d_1}\rightarrow \mathbb{R}^{{m_1} \times d_1}$ is given by $g_2^{(i_2,j_2)}(W_2) = C_2\tanh(W_2^{(i_2,j_2)}/C_2)$ with $C_1, C_2>0$. Here, we assume that the weight parameters are bounded, which is reasonable from both practical and analytical viewpoints (see \cite{rao1998function} and references therein). To achieve this, we apply $g_1, g_2$ to the weight matrices $W_1, W_2$, then, each element of $g(W_1), g(W_2)$ are bounded by $C_1, C_2$, respectively. We consider $g(z) = c\tanh(z/c)$, $z \in\mathbb{R}$, $c>0$ for the weight transformation as it is bounded, while 
for fixed $\epsilon>0$, $|g(z)-z|<\epsilon$ for $z \in (-I(c),I(c))$ with $0<I(c)<c$. Moreover, it is continuously differentiable with bounded first derivative and monotonically increasing on $\mathbb{R}$. Denote by $[W_1], [W_2]$ the vectors of all elements in $W_1, W_2$ respectively, and denote by $\theta = ([W_1], [W_2], b_1, b_2) \in \mathbb{R}^d$ with $d = 2d_1{m_1} +d_1+{m_1}$. We aim to minimize the regularized objective function:
\[
\min_{\theta}U(\theta) =\min_{\theta}\left(\E[|X-\widehat{X}|^2]+c|\theta|^2\right),
\]
where $c>0$. Autoencoders can be applied to extract the market implied features, see \cite{dixon2020machine, heaton2017deep}. A denoising autoencoder (DAE) with masking noise (see Subsection 3.3 of \cite{vincent2010stacked}) is considered in \cite{kirczenow2018machine}, which is used to learn the features of the missing yield parameters from the bond yields in a chosen surrogate liquid market, and thus to obtain missing bond yields in illiquid market. For the DAE algorithm, a noisy version $\widetilde{x} \in \mathbb{R}^{m_1}$ of the input vector $x$ is used in the neural network, which is obtained by applying the corruption process $q_D(\widetilde{x}|x)$. The objective function of the DAE algorithm becomes
\[
\min_{\theta}\overline{U}(\theta) =\min_{\theta}\left(\E[|X-\overline{X}|^2]+c|\theta|^2\right),
\]
where $\overline{x} = \sigma_2(g_2(W_2)\sigma_1(g_1(W_1)\widetilde{x}+b_1)+b_2)$.

In the online setting, the sequence of the (dependent) input vectors $\{x_n\}_{n \in \mathbb{Z}}$, which are bond yields in the liquid market, can be generated using the Nelson– Siegel model (see \cite{nelson1987parsimonious}, \cite{diebold2006forecasting} and \cite{jrfm13040065}). Moreover, the input datapoints are scaled with the maximum yield equal to one, and the corrupted version of each input $\{\widetilde{x}_n\}_{n \in \mathbb{Z}}$ is generated  according to the distribution $\widetilde{x}_i \sim q_D(\widetilde{x}_i|x_i)$ for each $i \in \mathbb{Z}$ before feeding into the neural network. The activation functions $\sigma_1$, $\sigma_2$ are set to be elementwise sigmoid function. Denote by $H: \mathbb{R}^d \times \mathbb{R}^m \rightarrow \mathbb{R}^d$ the stochastic gradient given by
\begin{align*}
H(\theta, z)
&= (H_{W_1^{(1,1)}}(\theta, z), \dots, H_{W_1^{(d_1,{m_1})}}(\theta, z), H_{W_2^{(1,1)}}(\theta, z), \dots, H_{W_2^{({m_1},d_1)}}(\theta, z), \\
&\qquad H_{b_1^{(1)}}(\theta, z), \dots,H_{b_1^{(d_1)}}(\theta, z),H_{b_2^{(1)}}(\theta, z), \dots,H_{b_2^{({m_1})}}(\theta, z)),
\end{align*}
where $z = (x, \widetilde{x}) \in \mathbb{R}^{m}$ with $m = 2m_1$. Then, one obtains, for any $i_1 = 1, \dots, d_1$, $j_1 = 1, \dots, {m_1}$,
\begin{align*}
&H_{W_1^{(i_1,j_1)}}(\theta, z) \\
& = 2cW_1^{(i_1,j_1)}-2\sum_{k = 1}^m(x^{(k)}-\overline{x}^{(k)})\partial_{W_1^{(i_1,j_1)}}\sigma_2^{(k)}(g_2^{(k, \cdot)}(W_2)\sigma_1(g_1(W_1)\widetilde{x}+b_1)+b_2^{(k)}) \\
&\quad \times g_2^{(k,i_1)}(W_2)\partial_{W_1^{(i_1,j_1)}}\sigma_1^{(i_1)}(g_1(W_1)^{(i_1,\cdot)}\widetilde{x}+b_1^{(i_1)})\partial_{W_1^{(i_1,j_1)}}g_1^{(i_1,j_1)}(W_1)\widetilde{x}^{(j_1)}.
\end{align*}
One can check that all of the assumptions hold for $H_{W_1^{(i_1,j_1)}}(\theta, z)$. Similarly, the assumptions hold for $H_{W_2^{(i_2,j_2)}}(\theta, z), H_{b_1^{(i_3)}}(\theta, z), H_{b_2^{(i_4)}}(\theta, z)$ with $i_2, i_4 = 1, \dots, {m_1}$, $j_2, i_3 = 1, \dots, d_1$.}
\end{example}

\appendix

\section{Auxiliary results}
We present a simpler version of  \cite[Theorem~7.19]{liptser:shiryaev:2001}, which is suitable for the purposes of this article.
\begin{lemma}  \label{lsKL}
Let $(\xi_t)_{t\ge0}$ and $(\eta_t)_{t\ge0}$ be two diffusion type processes with
\begin{equation}\label{ksi}
  \rmd \xi_t = a_t(\xi) \rmd t + \sigma \rmd B_t, \qquad \mbox{for } t>0,
\end{equation}
and
\begin{equation}\label{eta}
  \rmd \eta_t = b_t(\eta) \rmd t + \sigma \rmd B_t \qquad \mbox{for } t>0,
\end{equation}
where $\xi_0=\eta_0$ is an $\mathcal{F}_0$ measurable random variable and $\sigma*$ is a positive constant. Suppose also that the nonanticipative functionals $(a_t)_{t \geq 0}$ and $(b_t)_{t \geq 0}$ are such that a unique (continuous) strong solution exist for \eqref{ksi} and \eqref{eta} respectively. If, for any fixed  $T>0$,
\[
\int_0^T[|a_s(\xi)|^2 + |b_s(\xi)|^2] \rmd s <\infty \mbox{ (a.s.) and } \int_0^T[|a_s(\eta)|^2 + |b_s(\eta)|^2] \rmd s <\infty \mbox{ (a.s.),}
\]
then $\mu_\xi^{T}= \mathcal{L}(\xi_{[0,T]}) \backsim \mu_\eta^{T}= \mathcal{L}(\eta_{[0,T]})$ and the densities are given by
\begin{equation}\label{eq:density_eta_ksi}
  \frac{\rmd \mu^T_{\eta}}{\rmd \mu^T_{\xi}} (\xi) = \exp\left(-\sigma^{-2} \int_0^T \ps{a_s(\xi)-b_s(\xi)}{\rmd\xi_s} + \frac{1}{2\sigma^2}\int_0^T [|a_s(\xi)|^2 - |b_s(\xi)|^2]\rmd s\right)
\end{equation}
and
\begin{equation}\label{eq:density_eta_ksi1}
  \frac{\rmd \mu^T_{\xi}}{\rmd \mu^T_{\eta}} (\eta) = \exp\left(\sigma^{-2} \int_0^T \ps{a_s(\eta)-b_s(\eta)}{\rmd\eta_s} - \frac{1}{2\sigma^2}\int_0^T [|a_s(\eta)|^2 - |b_s(\eta)|^2]\rmd s\right).
\end{equation}
Finally, the Kullback-Leibler divergence is given by
\begin{equation}\label{eq:definition-KL}
  \operatorname{KL}(\mu_\xi^T,\mu_\eta^T ) = \frac{1}{2 \sigma^2}\E\left[\int_0^T|a_s(\xi) - b_s(\xi)|^2 \rmd s\right].
\end{equation}
\end{lemma}
\begin{proof}
The proof follows from a straightforward extension of \cite[Theorem~7.19]{liptser:shiryaev:2001}
to the multidimensional case. The computation of the Kullback-Leibler distance is a direct application of the definition.
\end{proof}
Let $V: \rset^d \to [1,\infty)$ be a measurable function. For a measurable function $f: \rset^d \to \rset$, the $V$-norm of $f$ is given by $\| f \|_V= \sup_{x \in \rset^d} |f(x)| / V(x)$. For $\xi$ and $\xi'$ two probability measures on $\rset^d$, the $V$-total variation distance of $\xi$ and $\xi'$ is given by
\[
\| \xi - \xi' \|_V= \sup_{\| f \|_V \leq 1}  \int_{\rset^d} f( \theta) \rmd \{ \xi - \xi' \}(\param)  .
\]
If $V \equiv 1$, then $\| \cdot \|_V$ is the total variation distance.
The $V$-total variation distance is also characterized in terms of coupling (see \cite[Theorem~19.1.7]{tome:moulines:2018}):
\[
\|\xi- \xi' \|_V = \inf_{\zeta\in\mathcal{C}(\xi,\xi')} \iint_{\rset^d \times \rset^d}
\{V(\theta) + V(\theta') \} \mathbbm{1}_{\{\theta \ne \theta'\}} \zeta(\rmd \theta, \rmd \theta')
\]
where $\mathcal{C}(\xi,\xi')$ is the set of coupling of $\xi$ and $\xi'$. An optimal coupling is given by
(see \cite[Theorem~19.1.6]{tome:moulines:2018})
\[
\gamma^*(B)= \{1 - \xi \wedge \xi'(\rset^d) \} \beta(B) + \int_B \xi \wedge \xi'(\rmd \theta) \delta_\theta(\rmd \theta')
\]
where $\xi \wedge \xi'$ is the infimum of probability measures $\xi$ and $\xi'$ and $\beta$ is any coupling of $\eta$ and $\eta'$ where
\[
\eta = \frac{\xi - \xi \wedge \xi'}{1 - \xi \wedge \xi'(\rset^d)} \quad \text{and} \quad
\eta' = \frac{\xi - \xi \wedge \xi'}{1 - \xi \wedge \xi'(\rset^d)}
\]
\begin{lemma}
\label{lem:domination-w-1}
For any probability measures $\xi$ and $\xi'$ on $\rset^d$, and $p \geq 1$, we get
\[
w_{1,p}(\xi,\xi')\leq \sqrt{2} \left\{ 1 +  [\xi(V_{2p})]^{1/2} + [\xi'(V_{2p})]^{1/2} \right\} \left\{ \operatorname{KL}(\xi,\xi') \right\}^{1/2} \, .
\]
\end{lemma}
\begin{proof}
\begin{align*}
w_{1,p}(\xi,\xi')
&= \inf_{\zeta \in \mathcal{C}(\xi,\xi')} \iint_{\rset^{2d}} (1 \wedge |\theta - \theta'|) \{1 + V_{p}(\theta) + V_p(\theta') \} \zeta(\rmd \theta \rmd \theta') \\
&\leq \iint_{\rset^{2d}} (1 \wedge |\theta - \theta'|) \{1 + V_p(\theta) + V_p(\theta')\} \gamma^*(\rmd \theta\rmd \theta') \\
&\leq \{1 - \xi \wedge \xi'(\rset^d)\} \iint_{\rset^{2d}} \{ 1 + V_p(\theta) + V_p(\theta') \} \beta(\rmd \theta \rmd \theta') \\
&= \|\xi - \xi'\|_{\operatorname{TV}} + \| \xi - \xi' \|_{V_p} \,.
\end{align*}
The proof then follows from the weighted  Pinsker's inequality; see \cite[Lemma~24]{durmus:moulines:2017}.
\end{proof}
\begin{lemma}\label{trivial}
	Let $x,\, y\in\rset^{d}$, then
$$
	\sum_{\substack{i+j+k=p \\ \{i\neq p-1\}\cap\{j\neq1\}}}\frac{p!}{i!j!k!}\|x\|^{2i}\big(2\ps{x}{y}\big)^j \|y\|^{2k} \le
\sum_{\substack{k=0 \\ k\neq 1}}^{2p}\binom{2p}{k}\|x\|^{2p-k}\|y\|^{k}
$$
\end{lemma}

\begin{proof}
Note that
\begin{align}\label{inequality}
 &\sum_{\substack{i+j+k=p \\ \{i\neq p-1\}\cap\{j\neq1\}}}\frac{p!}{i!j!k!}\|x\|^{2i}\big(2\langle x, y \rangle\big)^j \|y\|^{2k}
 &\le
  \sum_{\substack{i+j+k=p \\ \{i\neq p-1\}\cap\{j\neq1\}}}\frac{p!}{i!j!k!}\|x\|^{2i}\big(2\|x\| \|y\|\big)^j \|y\|^{2k}.
\end{align}

Moreover,
\begin{align*}
\sum_{k=0}^{2p}\binom{2p}{k}\|x\|^{2p-k}\|y\|^{k} =& (\|x\|+\|y\|)^{2p} = (\|x\|^2 + 2\|x\|\|y\| + \|y\|^2)^p\\ = &  \sum_{i+j+k=p}\frac{p!}{i!j!k!}\|x\|^{2i}\big(2\|x\| \|y\|\big)^j \|y\|^{2k}.
\end{align*}
Consequently,
\begin{align} \label{equality}
\sum_{\substack{k=0 \\ k\neq 1}}^{2p}\binom{2p}{k}\|x\|^{2p-k}\|y\|^{k} =&   \sum_{\substack{i+j+k=p \\ \{i\neq p-1\}\cap\{j\neq1\}}}\frac{p!}{i!j!k!}\|x\|^{2i}\big(2\|x\| \|y\|\big)^j \|y\|^{2k}.
\end{align}
Thus, in view of \eqref{inequality} and \eqref{equality}, the desired result is obtained.
\end{proof}

\end{document}